\documentclass[11pt]{amsart}
\usepackage{geometry} % see geometry.pdf on how to lay out the page. There's lots.
\usepackage{hyperref}
\usepackage{bbm}
\usepackage{xcolor}
\usepackage{comment}
\usepackage{amssymb,amsthm,amsmath,amsfonts,amsbsy,latexsym,dsfont}
\usepackage{enumerate}   
\usepackage{setspace}
\newtheorem{prop}{Proposition}
\newtheorem{thm}{Theorem}
\newtheorem{lemma}{Lemma}
\newtheorem{assume}{Assumption}
\newtheorem{remark}{Remark}
\newtheorem{definition}{Definition}
\newtheorem{corollary}{Corollary}

\numberwithin{prop}{section}
\numberwithin{thm}{section}
\numberwithin{lemma}{section}
\numberwithin{assume}{section}
\numberwithin{remark}{section}
\numberwithin{equation}{section}
\numberwithin{definition}{section}
\numberwithin{corollary}{section} 

\DeclareMathOperator*{\lambdargmin}{argmin}

\def \a {\alpha}
\def \E {\mathbb{E}}
\def \R {\mathbb{R}}

\def \cP {\mathcal{P}}
\def \wt {\widetilde}
\def \pa {\partial}
\def \cL {\mathcal{L}}
\def \cM {\mathcal{M}}

\def \cC {\mathcal{C}}
\def \W {\mathcal{W}}
\def \N {\mathbb{N}}
\def \ld {\lambda}

\def \ka{\kappa}

\def \ol{\overline}
\def \cF{\mathcal{F}}

\def \ul{\underline}

\def \k{\kappa}
\def \e{\epsilon}
\def \cM{\mathcal{M}}
\def \bs{\boldsymbol}
\def \h{\hat}

\makeatletter
\@namedef{subjclassname@2020}{%
  \textup{2020} Mathematics Subject Classification}

\title[]{Propagation of Chaos of Forward-Backward Stochastic Differential Equations with Graphon Interactions}
\author[]{Erhan Bayraktar} \thanks{E. Bayraktar is partially supported by the National Science Foundation under grant DMS2106556 and by the Susan M. Smith chair} 
\address{Department of Mathematics, University of Michigan}
\email{erhan@umich.edu}
\author[]{Ruoyu Wu}
\address{Department of Mathematics,  Iowa State University}
\email{ruoyu@iastate.edu}
\author[]{Xin Zhang} 
\address{Department of Mathematics, University of Vienna}
\email{xin.zhang@univie.ac.at}
\date{\today}

\begin{document}
\maketitle

\begin{abstract}
In this paper, we study graphon mean field games using a system of forward-backward stochastic differential equations. We establish the existence and uniqueness of solutions under two different assumptions and prove the stability with respect to the interacting graphons  which are necessary to show propagation of chaos results. As an application of propagation of chaos, we prove the convergence of $n$-player game Nash equilibrium for a general model, which is new in the theory of graphon mean field games. 
\end{abstract}

\begin{singlespace}
 \renewcommand{\contentsname}{\centering Contents}
{\footnotesize \tableofcontents}
\end{singlespace}

\section{Introduction}
The theory of mean field games was pioneered  independently by Lasry, Lions (see  \cite{MR2269875}, \cite{LASRY2006679}, \cite{MR2295621}) and Caines, Huang, Malham\'{e} (see  \cite{4303232}, \cite{MR2346927}). It is the study of strategic decision making by small interacting agents in large populations, and we refer the readers to \cite{BAYRAKTAR202198,2020arXiv201204845B,MR3860894,MR4082352,MR4013871} for  finite state mean field games, to \cite{MR4127851,MR4083905,MR3981375,MR4046528} for uniqueness of mean field game solutions, and to  \cite{MR3752669,MR3753660} for a nice survey.
Since then, the convergence of $n$-player game Nash equilibrium to the solution of mean field game has attracted lots of attention. There are three different ways to tackle this problem: 1) establish the regularity of solutions of the so-called Master equations, see \cite{MR3967062}; 2) use compactness arguments to show the existence of weak limits of $n$-player game Nash equilibrium, and prove any weak limit is a weak solution of mean field game solution, see \cite{MR3520014,MR4133381}; 3) prove it using propagation of chaos results for forward-backward stochastic differential equations or backward stochastic differential equations, see respectively \cite{2020arXiv200408351L,2021arXiv210500484P}.

In this paper, we investigate an analogous $n$-player game convergence problem for graphon mean field games. The standard mean field game theory assumes that interaction of different agents is symmetric. Recently, asymmetric graph connections among agents have been considered; see e.g. \cite{2021arXiv210512320A,2020arXiv200912144C,CainesPeterE2021GMFG,CarmonaCooneyGravesLauriere-2019-StoGraphonGamesStatic}. The heterogeneous interaction of players is modeled by graphons, which is a natural notion for the limit of a sequence of dense graphs. It was first introduced by Lov\'{a}sz et al.;  see e.g. \cite{MR2455626,borgs2012convergent,lovasz2006limits}, and have been used to characterize heterogeneously interacting particle systems; see e.g. \cite{2020arXiv200313180B,2020arXiv200810173B}, and also has important applications in $k$-core theory; see e.g. \cite{2020arXiv201209730B,MR2376426}. In graphon mean field games, given a graphon $G:[0,1] \times [0,1] \to \R_+$, each $\lambda \in [0,1]$ stands for a type of large subpopulation, and the correlation  between two subpopulations $\lambda$ and $\kappa$ is characterized by $G(\lambda, \kappa)$. As far as we know, the convergence problem for graphon mean field game has only been solved by \cite{2021arXiv210512320A} for a linear-quadratic model, where their argument heavily relied on the existence of explicit solutions. In this work, we study the limiting system which is a family of forward-backward stochastic differential equations (FBSDEs), and prove the $n$-player game convergence for a general model using propagation of chaos.

More specially, we investigate a coupled system of $n$ FBSDEs of the form
\begin{align}\label{eq:intro1}
\begin{cases}
	dX_t^{i,n}  =  \frac{1}{n} \sum_{j=1}^n G_n(\frac{i}{n},\frac{j}{n}) \h{B}(t,X_t^{i,n},X_t^{j,n},Y_t^{i,n})\,dt + \sigma\,dW_t^{i/n} \\
	dY_t^{i,n}  =  -\frac{1}{n} \sum_{j=1}^n G_n(\frac{i}{n},\frac{j}{n}) \h{F}(t,X_t^{i,n},X_t^{j,n},Y_t^{i,n})\,dt + \sum_{j=1}^n Z_t^{i,j,n}\,dW_t^{j/n}, \\
	X_0^{i,n}=\xi^{i/n}, \\
	Y_T^{i,n}  =  \frac{1}{n} \sum_{j=1}^n G_n(\frac{i}{n},\frac{j}{n}) \hat{Q}(X_T^{i,n},X_T^{j,n}),  \quad \quad i=1, \dotso, n, 
\end{cases}
\end{align}
where $\hat{B}, \hat{F}, \hat{Q}$ are drifts and terminal, $(\xi^{\ld})_{\ld \in [0,1]}$, $(W^{\ld})_{\ld \in [0,1]}$ are independent initial positions and Brownian motions respectively, and $G_n$ is a sequence of grahpons characterizing interactions between $(X^{i,n})_{i=1,\dotso,n}$. As $n \to \infty$ and $G_n \to G$, we show the propagation of chaos result that the above particle system converges to the following graphon interacting particle system 
\begin{align}\label{eq:intro2}
\begin{cases}
dX^{\lambda}_t = \int_0^1 \int_\R G(\lambda,\kappa) \h{B}\left(t,X^{\lambda}_t,x,Y^{\ld}_t \right) \cL(X_t^\kappa)(dx) \,d\kappa \, dt + \sigma \, d W^{\lambda}_t, \\
dY^{\lambda}_t = -\int_0^1 \int_\R G(\lambda,\kappa) \h{F}\left(t,X^{\lambda}_t,x,Y^{\ld}_t \right) \cL(X_t^\kappa)(dx) \,d\kappa \, dt +Z^{\lambda}_t \, d W^{\lambda}_t, \\
 X^{\lambda}_0 =\xi^{\lambda}, \\
 Y^{\lambda}_T =  \int_0^1 \int_{\R} G(\lambda,\kappa) \hat{Q}(X_T^{\lambda},x)\,\cL(X_T^{\kappa})(dx)\,d\kappa, \quad \quad \ld \in [0,1],
\end{cases}
\end{align}
where $\cL(X^{\k}_t)$ denotes the law of $X^{\k}_t$. By the stochastic maximum principle, the solution of the graphon mean field game can be characterized using \eqref{eq:intro2}.  One can conclude the convergence of $n$ player game using the convergence  $\eqref{eq:intro1} \Rightarrow \eqref{eq:intro2}$.

To carry out our analysis, first we study the existence and uniqueness of the limiting system \eqref{eq:intro2}. To properly define the interaction term $$ \int_0^1 \int_\R G(\lambda,\kappa) \h{B}\left(t,X^{\lambda}_t,x,Y^{\ld}_t \right) \cL(X_t^\kappa)(dx) \,d\kappa \, dt ,$$ we must show the measurability of $\ld \to \cL(X^{\ld}_t)$. Since we are only interested in the marginals $\cL(X^{\ld}_t)_{\ld \in [0,1]}$ instead of the joint law $\cL(X_t)$, one can study another FBSDE system where all the $(X^{\ld}, Y^{\ld})$ are driven by the same Brownian motion. The marginal laws of these two systems are the same which is a useful observation in establishing measurability.  Then under two commonly used monotonicity assumptions as in \cite{2021arXiv210209619B}, we show the existence and uniqueness of solutions to \eqref{eq:intro2}. Next we study the stability of solutions with respect to the interacting graphon $G$. As a natural notion of distance between two different graphons, the cut norm is widely used and is weaker than the $L^1$-norm; see e.g. \cite{lovasz2006limits}. To make an estimation involving the cut norm, we adopt the argument of \cite[Theorem 2.1]{2020arXiv200313180B} where the boundedness of solutions in the $L^p$ norm for $p >2$ is necessary. This is the reason that we prove existence results in general $(L^p)_{p \geq 2}$ spaces. Finally, we prove the propagation of chaos for FBSDEs, where the stability is essential. Assuming that the cut norm convergence of the interacting graphons $\lVert G_n - G \rVert _{\square} \to 0$, we show that $\eqref{eq:intro1} \Rightarrow \eqref{eq:intro2}$. Under a stronger condition that $G_n$ is the uniform block sampling of $G$, we can obtain the convergence rate.

The rest of the paper is organized as follows. In Section 2, we prove the unique existence of solutions of \eqref{eq:gGMFG} under two different assumptions. In Section 3, assuming that the interaction is linear in $G$, we show the stability of solutions with respect to $G$. In Section 4, we prove the Propagation of Chaos result, and in section 5 we apply previous results to a simple model of graphon mean field game.

\subsection{Notation} Denote by $\cM( [0,T]; \cP_p(\R))$ and $\cC([0,T]; \cP_p(\R))$ the space of measurable functions from $[0,T]$ to $\cP_p(\R)$ and the space of continuous functions from $[0,T]$ to $\cP_p(\R)$ respectively. Define the space of families of probability flows as 
\begin{align*}
P F^p:=\{\bs{\mu}: [0,1] \to \cC\left([0,T];\cP_p(\R)\right): \, \lambda \mapsto \bs{\mu}^{\lambda} \text{ is measurable}\}.
\end{align*}
For any $\mu, \wt{\mu} \in \cP(\mathcal{C}([0,T];\R^2))$, let us take 
$$\W_{2,T}(\mu, \wt {\mu}):= \inf\left\{  \E\left[\sup_{t\in[0,T]} |X_t-\wt{X}_t|^2+\sup_{t\in[0,T]} |Y^{\ld}_t-\wt{Y}_t|^2\right] : \, (X,Y) \sim \mu, \, (\wt{X},\wt{Y}) \sim \wt{\mu} \right\}.$$ 
For a family of random variables $\{X^{\ld}\}_{\ld \in [0,1]}$, we denote by $\cL^m(X)$ the set of laws $\{\cL(X^{\ld})\}_{\ld \in [0,1]}$.

We define spaces of processes and random variables
\begin{itemize}
\item $L^{p,2}_{\mathcal{F}}$ to be the set of all $\{\mathcal{F}_t\}_{t \geq 0}$- progressively measurable real-valued process $(X_t)_{t \geq 0}$ such that $\E \left[ \left(\int_0^T |X_t|^2 \, dt \right)^{p/2}\right]< + \infty$. 
\item $L^{p,c}_{\mathcal{F}}$ to be the set of all $\{\mathcal{F}_t\}_{t \geq 0}$- progressively measurable real-valued continuous process $(X_t)_{t \geq 0}$ such that $\E\left[  \sup_{t \in [0,T]}  |X_t|^p \right]< + \infty$. 
\item $L^p_{\mathcal{F}_t}$ to be the set of all $\mathcal{F}_t$-measurable $p$-th integrable random variables. 
\item $\cM L^{p,2}_{\mathcal{F}}$ to be the set of all measurable functions $X$ from $[0,1]$ to $L^{p,2}_{\mathcal{F}}$ such that $\max_{\ld \in [0,1]} \E \left[ \left( \int_0^T |X^{\ld}_t|^2 \, dt \right)^{p/2} \right]< + \infty$ , and similarly define $\cM L^{p,c}_{\mathcal{F}}$, $\cM L^p_{\mathcal{F}_t}$.

\end{itemize}

For any $x \in \cM L^{p,c}_{\mathcal{F}}$, we define norms
\begin{align*}
\lVert x \rVert^{I,p}_{k} :=&\E\left[ \int_{[0,1]}    \int_0^T e^{k t} \left( x^{\ld}_t \right)^p \, dt  \, d \ld \right], \\ 
\lVert x \rVert^p_{k} :=& \max_{\ld \in [0,1]}\E\left[  \int_0^T e^{k t} \left( x^{\ld}_t \right)^p \, dt \right], \\ 
\lVert x \rVert^p_S:=& \max_{\ld \in [0,1]} \sup_{t \in [0,T]} \E\left[  |x^{\ld}_t|^p \right]. 
\end{align*}
For any $x \in L^{p,c}_{\mathcal{F}}$, define its norm $$\wt{\lVert x \rVert}^p_S:= \sup_{t \in [0,T]} \E[ |x_t|^p].$$

\section{Unique existence of solutions}

In this section, we show the existence and uniqueness of solutions to a general FBSDE system \eqref{eq:gGMFG} under two assumptions. 
\begin{align}\label{eq:gGMFG}
\begin{cases}
dX^{\lambda}_t = {B}_G^{\lambda}\left(t,X^{\lambda}_t,\cL^m(X_t),Y^{\ld}_t \right) \, dt + \sigma \, d W^{\lambda}_t, \\
dY^{\lambda}_t = -F_G^{\lambda}\left(t,X^{\lambda}_t,\cL^m(X_t), Y^{\lambda}_t\right) dt +Z^{\lambda}_t \, d W^{\lambda}_t, \\
 X^{\lambda}_0 =\xi^{\lambda},  \\
 Y^{\lambda}_T =Q_G^{\lambda}\left(X^{\lambda}_T,\cL^m(X_T)\right), \quad \forall \ld \in [0,1].
\end{cases}
\end{align}
The proof is very similar to that of McKean Vlasov FBSDE, except the verification of measurability of $\ld \mapsto \cL(X^{\ld})$. Our main observation is that different types of players ($X^{\ld})$ are only interacting through their marginal laws $\cL^m(X)$, and therefore by the weak uniqueness of FBSDE solutions, we can treat all processes $(X^{\ld}, Y^{\ld},Z^{\ld})_{\ld \in [0,1]}$ on one stochastic basis. Then we prove a stronger statement that $\ld \mapsto X^{\ld}$ is measurable in the $L^2$ sense, which implies the measurability of $\ld \mapsto \cL(X^{\ld})$.

\begin{definition}
A family of processes $(X^{\ld}, Y^{\ld}, Z^{\ld})_{ \ld \in [0,1]}$ is said to be a solution of \eqref{eq:gGMFG} if $ \ld \mapsto \cL(X^{\ld})$ is measurable and $(X^{\ld}, Y^{\ld}, Z^{\ld})$ satisfies the FBSDE system \eqref{eq:gGMFG} for each $\ld \in [0,1]$. 
\end{definition}
For simplicity of notation, we suppress $G$ when the graphon is clear from the context. We make the following two assumptions in this section.

\begin{assume}\label{assume1}
(i) $B^{\ld}$ is Lipschitz in $x$, and there exists a constant $K_1 \in \R$ such that for any $\ld \in [0,1]$, $(t,x,x',y) \in [0,T] \times \R^3$, $\bs{\eta} \in \cM([0,1]; \cP_p(\R))$  
\begin{align*}
(x-x') \cdot \left( B^{\ld}(t,x,\bs{\eta},y) -B^{\ld} (t,x',\bs{\eta},y) \right) \leq -K_1 (x-x')^2
\end{align*}
(ii) $F^{\ld}$ is Lipschitz in $y$, and there exists a constant $K_2 \in \R $ such that for any $\ld \in [0,1]$, $(t,x,y,y') \in [0,T] \times \R^3$, $\bs{\eta} \in \cM([0,1]; \cP_2(\R))$
\begin{align*}
(y-y') \cdot \left( F^{\ld}(t,x,\bs{\eta},y) -F^{\ld} (t,x,\bs{\eta},y') \right) \leq -K_2 (y-y')^2
\end{align*}
(iii) $B^{\ld}$ is $L_1$-Lipschitz in $y$, $F^{\ld}$ is $L_2$-Lipschitz in $x$, $Q^{\ld}$ is $L_3$-Lipschitz in $x$, and it holds that 
 \begin{align*}
 |B^{\ld}(t,x,\bs{\eta},y)-B^{\ld}(t,x,\wt{\bs{\eta}},y)| &\leq \frac{L_1}{2} \W_p(\bs{\eta}^{\ld}, \wt{\bs{\eta}}^{\ld})+ \frac{L_1}{2} \int_{[0,1]} \W_p(\bs{\eta}^{\ld},\wt{\bs{\eta}}^{\ld}) \, d \ld, \\
  |F^{\ld}(t,x,\bs{\eta},y)-F^{\ld}(t,x,\wt{\bs{\eta}},y)| &\leq \frac{L_2}{2} \W_p(\bs{\eta}^{\ld}, \wt{\bs{\eta}}^{\ld})+ \frac{L_2}{2} \int_{[0,1]} \W_p(\bs{\eta}^{\ld},\wt{\bs{\eta}}^{\ld}) \, d \ld, \\
   |Q^{\ld}(x,\bs{\eta})-Q^{\ld}(x,\wt{\bs{\eta}})| &\leq \frac{L_3}{2} \W_p(\bs{\eta}^{\ld}, \wt{\bs{\eta}}^{\ld})+ \frac{L_3}{2} \int_{[0,1]} \W_p(\bs{\eta}^{\ld},\wt{\bs{\eta}}^{\ld}) \, d \ld.
 \end{align*}
(iv) It holds that $pK_1 +p K_2> (2p-1)L_1+(2p-2)L_2$ and there exists a constant $k \in ((2p-2)L_2-pK_2 ,p K_1-(2p-1)L_1 )$ such that 
\begin{align}\label{eq:contraction}
& (k+pK_2-(2p-2)L_2)  > \left(2^pL_1L_3^p+\frac{2^{p-1}L_1^2L_3^p+2L_1L_2}{-k+pK_1-(2p-1)L_1} \right).\end{align}
(v)  We have that $\lambda \mapsto \mathcal{L}(\xi^\lambda)$ is  measurable, $\sup_{\ld \in [0,1]} \E[|\xi^{\ld}|^p] < +\infty$, and $(B(\cdot, 0), F(\cdot, 0), Q) \in \cM L^{p,2}_{\cF} \times \cM L^{p,2}_{\cF} \times \cM L^p_{\cF_T}$. 
\end{assume}

\begin{assume}\label{assume2}
(i) $(B^{\ld},F^{\ld})$ are $L$-Lipschitz in $(x,y)$, and it holds that 
 \begin{align*}
 |B^{\ld}(t,x,\bs{\eta},y)-B^{\ld}(t,x,\wt{\bs{\eta}},y)| &\leq l \W_2(\bs{\eta}^{\ld}, \wt{\bs{\eta}}^{\ld})+ l \int_{[0,1]} \W_2(\bs{\eta}^{\ld},\wt{\bs{\eta}}^{\ld}) \, d \ld, \\
  |F^{\ld}(t,x,\bs{\eta},y)-F^{\ld}(t,x,\wt{\bs{\eta}},y)| &\leq l \W_2(\bs{\eta}^{\ld}, \wt{\bs{\eta}}^{\ld})+ l \int_{[0,1]} \W_2(\bs{\eta}^{\ld},\wt{\bs{\eta}}^{\ld}) \, d \ld, \\
 |Q^{\ld}(x,\bs{\eta})-Q^{\ld}(x,\wt{\bs{\eta}})| &\leq l \W_2(\bs{\eta}^{\ld}, \wt{\bs{\eta}}^{\ld})+ l \int_{[0,1]} \W_2(\bs{\eta}^{\ld},\wt{\bs{\eta}}^{\ld}) \, d \ld.
 \end{align*}
(ii) There exist a positive constant $k>3l$ such that for all $ \ld \in [0,1]$,
\begin{align*}
& -\Delta x^{\ld} \left(F^{\ld}(t,\theta^{\ld})-F^{\ld}(t,\wt{\theta}^{\ld}) \right) +\Delta y^{\ld}\left( B^{\ld}(t,\theta^{\ld})-B^{\ld}(t,\wt{\theta}^{\ld})  \right) \\ 
& \leq -k   (\Delta x^{\ld})^2-k(\Delta y^{\ld})^2 , \\
&  \Delta x^{\ld}\left( Q^{\ld}(x^{\ld}, \bs{\eta})-Q^{\ld}(\wt{x}^{\ld}, \bs{\eta}) \right) \\
& \geq k (\Delta x^{\ld})^2,
\end{align*}
where $\Delta x^{\ld}:= x^{\ld}-\wt{x}^{\ld}$, $\Delta y^{\ld}:= y^{\ld}-\wt{y}^{\ld}$, $\theta^{\ld}=(x^{\ld}, \bs{\eta}, y^{\ld})$,$\wt{\theta}^{\ld}=(\wt{x}^{\ld}, \bs{\eta}, \wt{y}^{\ld})$. \\
(iii) We have that $\lambda \mapsto \mathcal{L}(\xi^\lambda)$ is  measurable, $\sup_{\ld \in [0,1]} \E[|\xi^{\ld}|^p] < +\infty$, and $(B(\cdot, 0), F(\cdot, 0), Q) \in \cM L^{p,2}_{\cF} \times \cM L^{p,2}_{\cF} \times \cM L^p_{\cF_T}$. 
\end{assume}

\begin{lemma}\label{lem:weakunique}
Under Assumption \ref{assume1} or \ref{assume2}, there is a one-to-one correspondence between solutions to \eqref{eq:gGMFG} and 
\begin{align}\label{eq:GMFG2}
\begin{cases}
dX^{\ld}_t= B_G^{\ld}\left(t,X^{\ld}_t,\cL^m(X_t), Y^{\ld}_t\right) \, dt + \sigma \, dW_t, \\
 dY^{\ld}_t= -F_G^{\ld} \left(t, X^{\ld}_t,\cL^m(X_t), Y^{\ld}_t \right) \, dt + Z^{\ld}_t \, dW_t, \\
 X^{\ld}_0= \xi^{\ld}, \\
 Y^{\ld}_T= Q_G^{\ld}\left(X^{\ld}_T,\cL^m(X_T)\right), \quad \forall \ld \in [0,1].
\end{cases}
\end{align}
\end{lemma}
\begin{proof}
Given a solution $(X,Y,Z)$ to \eqref{eq:gGMFG}, we plug in the law $\bs{\mu}(t)=\{ \cL(X_t^{\ld})\}_{\ld \in [0,1]}$ into the FBSDE for each $\ld \in [0,1]$
\begin{align}\label{eq:weakunique}
\begin{cases}
d\wt{X}^{\ld}_t= B^{\ld}\left(t,\wt{X}^{\ld}_t, \bs{\mu}(t), \wt{Y}^{\ld}_t\right) \, dt + \sigma \, dW_t, \\
 d\wt{Y}^{\ld}_t= -F^{\ld} \left(t, \wt{X}^{\ld}_t, \bs{\mu}(t), \wt{Y}^{\ld}_t \right) \, dt + \wt{Z}^{\ld}_t \, dW_t, \\
\wt{ X}^{\ld}_0= \wt{\xi}^{\ld}, \\
 \wt{Y}^{\ld}_T= Q^{\ld}\left(\wt{X}^{\ld}_T, \bs{\mu}(T) \right).
\end{cases}
\end{align}
It is well-known that there exists a pathwise unique solution to FBSDE \eqref{eq:weakunique} for each $\ld \in [0,1]$ under Assumptions \ref{assume1} or \ref{assume2} (see e.g. \cite{Hu:1995aa}, \cite{Pardoux:1999aa}).  Due to Lemma~\ref{lem:unique}, its law $\cL(\wt{X}^{\ld}_t)$ coincides with $\cL(X^{\ld}_t)$, and hence $\cL(\wt{X}_t)=\bs{\mu}(t)$. Therefore, the triple $(\wt{X},\wt{Y},\wt{Z})$ solves \eqref{eq:GMFG2}. The proof for the converse is the same. 
\end{proof}

Note that in \eqref{eq:GMFG2} there is only one driven Brownian motion $W$ for all $\ld \in [0,1]$, and thus one can work on one stochastic basis and prove measurability more conveniently. As a result of Lemma~\ref{lem:weakunique}, it is equivalent to solve \eqref{eq:gGMFG} and \eqref{eq:GMFG2}, and thus we study only \eqref{eq:GMFG2} in the remaining of this section.

\subsection{Contraction mapping}

We will prove there exists a unique solution to \eqref{eq:GMFG2} using contraction mapping theorem under Assumption~\ref{assume1} with a constant $p \geq 2$.

Given any measurable $y \in \cM L^{p,c}_{\mathcal{F}}$, we define $\Psi(y):=x$ as the unique solution to
\begin{align}\label{eq:step2}
\begin{cases}
dx^{\ld}_t=B^{\ld}(t, x^{\ld}_t, \cL^m(x_t), y^{\ld}_t) \, dt + \sigma \, d W_t,  \\
x^{\ld}_0=\xi^{\ld},  \quad \forall \ld \in [0,1].
\end{cases}
\end{align}
For any $x\in \cM L^{p,c}_{\mathcal{F}}$, define $\Phi(x):=y $ to be the unique solution to backward stochastic equations 
\begin{align}\label{eq:step1}
\begin{cases}
dy^{\ld}_t=-F^{\ld}(t, x^{\ld}_t, \cL^m(x_t), y^{\ld}_t) \, dt +z^{\ld}_t \, d W_t, \\
y^{\ld}_T=Q^{\ld}(x^{\ld}_T, \cL^m(x_T)),  \quad \forall \ld \in [0,1].
\end{cases}
\end{align}

\begin{lemma}\label{lem:picard}
For $y \in \cM L^{p,c}_{\mathcal{F}}$, there exists a unique solution $x$ to \eqref{eq:step2} such that $\ld \to x^{\ld}$ is measurable, and for $x \in \cM L^{p,c}_{\mathcal{F}}$, the solution $y= \Phi(x)$ to \eqref{eq:step1} is measurable with respect to $\ld$.
\end{lemma}
\begin{proof}
The proof is based on Picard iterations. 

\vspace{10pt}

\textit{Step 1:} For any $\bs{\mu} \in PF^p$, the following stochastic differential equations can be solved 
\begin{align}\label{step2:eq1}
\begin{cases}
dx^{\ld}_t=B^{\ld}(t, x^{\ld}_t, \bs{\mu}_t, y^{\ld}_t) \, dt + \sigma \, d W_t, \\
x^{\ld}_0=\xi^{\ld},
\end{cases}
\end{align}
via Picard iteration. Take $\wt{x} \in \cM L^{p,c}_{\mathcal{F}}$, and define 
\begin{align*}
\ol{x}^{\ld}_t=\xi^{\ld}+ \int_0^t B^{\ld}(s, \wt{x}^{\ld}_s, \bs{\mu}_t, y_t^{\ld}) \, ds + \sigma W_t.
\end{align*}
As a result of Lemma~\ref{lem:measurable'} and standard estimations of SDE,  it is clear that $\ol{x} \in \cM L^{p,c}_{\mathcal{F}}$. Also $\wt{x} \mapsto \ol{x}$ is a contraction under the norm $\lVert \cdot \rVert^p_{\a}$ with some $-\a$ large enough. Thus its fixed point $x$ solves \eqref{step2:eq1}, and $x \in \cM L^{p,c}_{\mathcal{F}}$.

 Take $\wt{x} \in \cM L^{p,c}_{\mathcal{F}}$, plug in $\bs{\mu}=\cL^m(\wt{x})$ into \eqref{step2:eq1}, and obtain its solution $x:=\Gamma(\wt{x})$. By a modification of \cite[Theorem 1.7]{MR3629171}, it can be easily shown that $\Gamma^{k}(\wt{x}) \in \cM L^{p,c}_{\mathcal{F}}$ converges, and its limit solves \eqref{eq:step2} and belongs to $\cM L^{p,c}_{\mathcal{F}}$. 

\vspace{10pt}

\textit{Step 2:} Take any $\wt{y} \in \cM L^{p,c}_{\mathcal{F}}$, denote $f^{\ld}(t):=F^{\ld}(t,x^{\ld}_t,\cL^m(x_t), \wt{y}^{\ld}_t), \, \forall \ld \in [0,1]$. We define 
\begin{align*}
\ol{y}^{\ld}_t:=\E\left[ Q^{\ld}(x^{\ld}_T,\cL^m(x_T)) - \int_t^T f^{\ld}(s) \, ds \, \Big| \,\mathcal{F}_t \right]
\end{align*}

By a modification of \cite[Theorem 2.2]{MR3629171}, it can be easily shown that $\wt{y} \mapsto \ol{y}$ is a contraction under the norm $\lVert \cdot \rVert^p_{\a}$ for some $\a$ large enough. Then the unique fixed point is actually the solution to \eqref{eq:step1}. For the measurability of $\ld \to \ol{y}^{\ld} \in L^{2,c}_{\mathcal{F}}$, due to Lemma~\ref{lem3} it suffices to show that $\ld \to \ol{y}^{\ld}_t \in L^p_{\mathcal{F}_t}$ is measurable for any $t$. By Jensen's inequality, it is readily seen that 
\begin{align}\label{conditional}
L^p_{\mathcal{F}_T} \ni \xi \mapsto \E[ \xi \, | \, \mathcal{F}_t] \in L^p_{\mathcal{F}_t}
\end{align}
is a contraction and thus is continuous. Due to Lemma~\ref{lem:measurable'}, $\ld \mapsto Q^{\ld}(x^{\ld}_T,\cL^m(x_T)) - \int_t^T f^{\ld}(s) \in L^p_{\mathcal{F}_T}$ is measurable. Therefore its composition with \eqref{conditional}, $\ld \mapsto \ol{y}_t^{\ld} \in L^p_{\mathcal{F}_t}$,  is measurable.

\end{proof}

Let us prove that $\Phi \circ \Psi$ is a contraction, and thus the unique fixed point is the unique solution to \eqref{eq:GMFG2}. The proof is the same as in \cite{Pardoux:1999aa}.

\begin{thm}\label{thm:contraction}
Under Assumption~\ref{assume1}, the composition $\Phi \circ \Psi$ is a contraction under the norm $\lVert \cdot \rVert_{k}^p$.
\end{thm}
\begin{proof}
Take $y, \, \wt{y}$, $x=\Psi(y), \, \wt{x}=\Psi(\wt{y})$, and $Y=\Phi(x), \, \wt{Y}=\Phi(\wt{x})$. Denote  $\Delta y^{\ld}_t = y^{\ld}_t -\wt{y}^{\ld}_t$, $\Delta x^{\ld}_t = x^{\ld}_t -\wt{x}^{\ld}_t$, $\Delta Y^{\ld}_t = Y^{\ld}_t -\wt{Y}^{\ld}_t$. By It\^{o}'s formula, we obtain that 
\begin{align*}
e^{kt} |\Delta x^{\ld}_t|^p =& k \int_0^t e^{ks} |\Delta x^{\ld}_s|^p \,ds \\
&+ p \int_0^t e^{ks} |\Delta x^{\ld}_s|^{p-2}  \Delta x^{\ld}_s \left(B^{\ld}(t,x^{\ld}_s,\cL^m(x_s),y^{\ld}_s)-B^{\ld}(t,\wt{x}^{\ld}_s,\cL^m(\wt{x}_s),\wt{y}^{\ld}_s) \right)     \,ds.
\end{align*}
According to Assumption~\ref{assume1}, Young's inequality and property of Wasserstein metric, we get that
\begin{align*} 
& |\Delta x^{\ld}_s|^{p-2} \Delta x^{\ld}_s \left(B^{\ld}(t,x^{\ld}_s,\cL^m(x_s),y^{\ld}_s)-B^{\ld}(t,\wt{x}^{\ld}_s,\cL^m(\wt{x}_s),\wt{y}^{\ld}_s) \right)  \\
 & \leq -K_1|\Delta x^{\ld}_s|^p+ \frac{L_1}{2} \left(\frac{|\Delta x^{\ld}_s|^{p}}{q}+ \frac{\E[|\Delta x^{\ld}_s |^p] ] }{p}\right) \\
 & \quad + \frac{L_1}{2} \left(\frac{|\Delta x^{\ld}_s|^{p}}{q}+\frac{\int_{[0,1]} \E[|\Delta x^{\ld}_s |^p] ]  \, d\k}{p} \right)+L_1\left(\frac{|\Delta x^{\ld}_s|^{p}}{q}+\frac{|\Delta y^{\ld}_s|^{p}}{p}\right),
\end{align*}
where $q$ is the conjugate of $p$, i.e., $\frac{1}{q}+\frac{1}{p}=1$. 
Thus we have 
\begin{align}\label{eq:contract1}
&e^{kt} \E[ |\Delta x^{\ld}_t|^p] + (-k+pK_1-(4p-3)L_1/2)\int_0^t e^{ks} \E[|\Delta x^{\ld}_s|^p ] \, ds \notag \\
&-\frac{L_1}{2} \int_0^t e^{ks} \int_{[0,1]} \E[ |\Delta x^{\k}_s|^p] \, d \k \, ds  \leq L_1 \int_0^t \E[|\Delta y^{\ld}_s|^p ] \, ds,
\end{align}
and hence 
\begin{align}\label{eq:contract5}
  \max_{\ld \in [0,1]} \int_0^t e^{k s}  \E [|\Delta x^{\ld}_s|^p ] \, ds   \leq \frac{L_1}{-k+pK_1-(2p-1)L_1} \max_{\ld \in [0,1]} \int_0^t e^{k s}   \E[ |\Delta y^{\ld}_s|^p] \, ds.
\end{align}
As a result of $k <pK_1-(2p-1)L_1$, one can deduce from \eqref{eq:contract1} that 
\begin{align}\label{eq:contract3}
& (-k+pK_1-(2p-1)L_1)  \int_0^t e^{k s} \int_{[0,1]} \E [|\Delta x^{\k}_s|^p ] \, d\k \, ds \notag \\
& \leq L_1 \int_0^t e^{k s} \int_{[0,1]}   \E[ |\Delta y^{\k}_s|^p] \, d \k \, ds \leq \max_{\k \in [0,1]} L_1 \int_0^t e^{ks} \E [|\Delta y^{\k}_s|^p] \, ds,
\end{align}
and also 
\begin{align}\label{eq:contract4}
e^{k t}  \E[|\Delta x^{\ld}_t|^p] \leq& L_1 \int_0^t e^{k s}   \E[ |\Delta y^{\ld}_s|^p] \, ds +L_1/2  \int_0^t e^{k s} \int_{[0,1]} \E [|\Delta x^{\k}_s|^p ] \, d\k \, ds. 
\end{align}
Taking maximum over $\ld$ in \eqref{eq:contract4} and using the inequality \eqref{eq:contract3}, we obtain that 
\begin{align}\label{eq:contract6}
& \max_{\ld \in [0,1]} e^{k t}  \E[|x^{\ld}_t -\wt{x}^{\ld}_t|^2] \notag \\
& \leq \left(L_1+\frac{L_1^2}{2(-k+pK_1-(2p-1)L_1)} \right) \max_{\ld \in [0,1]} \int_0^t e^{ks} \E [|y^{\ld}_s -\wt{y}^{\ld}_s|^2] \, ds.
\end{align}

For BSDEs, it can be easily seen that 
\begin{align*}
&e^{kt}|\Delta Y^{\ld}_t|^p = e^{kT}|Q^{\ld}(x^{\ld}_T,\cL^m(x_T))-Q^{\ld}(\wt{x}^{\ld}_T, \cL^m(\wt{x}_T)) |^p \\
&-k\int_t^T e^{ks} |\Delta Y^{\ld}_s| \, ds -p(p-1)/2 \int_t^T e^{ks} |\Delta Y^{\ld}_s|^{p-2}|\Delta Z^{\ld}_s|^2 \, ds  \\
&+p \int_t^T e^{ks} |\Delta Y_s^{\ld}|^{p-2}\Delta Y_s^{\ld} \left(F^{\ld}(s, x^{\ld}_s, \cL^m(x_s), Y^{\ld}_s)-F^{\ld}(s, \wt{x}^{\ld}_s, \cL^m(\wt{x}_s), \wt{Y}^{\ld}_s) \right) ds \\
&-p \int_t^T e^{ks} |\Delta Y_s^{\ld}|^{p-2}\Delta Y_s^{\ld} \Delta Z^{\ld}_s \, dW_s.
\end{align*}
Using Assumption~\ref{assume1}, Young's inequality and properties of Wasserstein metric, we get that 
\begin{align*}
 &\E\left[ |\Delta Y_s^{\ld}|^{p-2}\Delta Y_s^{\ld} \left(F^{\ld}(s, x^{\ld}_s, \cL^m(x_s), Y^{\ld}_s)-F^{\ld}(s, \wt{x}^{\ld}_s, \cL^m(\wt{x}_s), \wt{Y}^{\ld}_s) \right) \right]\\
 & \leq -K_2 \E[|\Delta Y^{\ld}_s|^p ]+ \frac{L_2}{2} \left(\frac{\E[|\Delta Y^{\ld}_s|^{p}]}{q}+ \frac{\E[|\Delta x^{\ld}_s |^p] ] }{p}\right) \\
 & \quad + \frac{L_2}{2} \left(\frac{\E[|\Delta Y^{\ld}_s|^{p}]}{q}+\frac{\int_{[0,1]} \E[|\Delta x^{\ld}_s |^p] ]  \, d\k}{p} \right)+L_2\left(\frac{\E[|\Delta Y^{\ld}_s|^{p}]}{q}+\frac{\E[|\Delta x^{\ld}_s|^{p}]}{p}\right).
\end{align*}
Therefore, one can obtain 
\begin{align*}
& e^{kt} \E [ |\Delta Y^{\ld}_t|^p]+ \left(k+pK_2 -(2p-2)L_2 \right) \int_t^T e^{ks} \E [ |\Delta Y^{\ld}_s|^p] \, ds  \\
& \leq (2^{p-2}+2^{p-1})L_3^pe^{kT} \E[|\Delta x_T^{\ld}|^p]+2^{p-2}L_3^p e^{kT} \int_{[0,1]} \E[|\Delta x_T^{\k}|^p ]\,d\k \\
&+3L_2/2 \int_t^T e^{ks} \E[ |\Delta x^{\ld}_s|^p ] \, ds +L_2 /2 \int_t^T e^{k s}  \int_{ [0,1]} E[ |\Delta x^{\k}_s|^p] \, d \k \, ds
\end{align*}
Plugging in \eqref{eq:contract5} and \eqref{eq:contract6}, it can be readily seen that 
\begin{align*}
& (k+pK_2-(2p-2)L_2) \lVert Y- \wt{Y} \rVert^p_{k}  \leq \left(2^pL_1L_3^p+\frac{2^{p-1}L_1^2L_3^p+2L_1L_2}{-k+pK_1-(2p-1)L_1} \right) \lVert y -\wt{y} \rVert^p_{k},
\end{align*}
and $\Phi \circ \Psi$ a contraction due to \eqref{eq:contraction}.

\end{proof}

\subsection{Method of Continuation} 
We consider the following family of FBSDEs parametrized by $\zeta \in [0,1]$,
\begin{align}\label{eq:conti}
\begin{cases}
dX^{\zeta,\ld}_t= \left(\zeta B^{\ld}(t,X^{\zeta,\ld}_t,\cL^m(X^{\zeta}_t),Y^{\zeta,\ld}_t)-(1-\zeta) Y^{\zeta, \ld}_t+B^{\ld}_0(t) \right) dt+ \sigma \, d W_t \\
dY^{\zeta,\ld}_t= -\left(\zeta Y^{\ld}(t,X^{\zeta,\ld}_t,\cL^m(X^{\zeta}_t),Y^{\zeta,\ld}_t)+(1-\zeta)X^{\zeta, \ld}_t+F^{\ld}_0(t) \right) dt+ Z^{\zeta,\ld}_t \, d W_t, \\
X^{\zeta,\ld}_0=\xi^{\ld}, \\
Y^{\zeta, \ld}_T=\zeta Q^{\ld}(X^{\zeta,\ld}_T, \cL^m(X^{\zeta,\ld}_T))+(1-\zeta) X^{\zeta,\ld}_t+Q^{\ld}_0,
\end{cases}
\end{align}
where $B_0, F_0, Q_0 \in \cM L^{p,2}_{\cF} \times \cM L^{p,2}_{\cF} \times \cM L^p_{\mathcal{F}_T}$. In the case that $\zeta=1, B_0^{\ld}=F_0^{\ld}=Q^{\ld}_0=0$, \eqref{eq:conti} reduces to \eqref{eq:gGMFG}. In the case that $\zeta=0$, \eqref{eq:conti} becomes 
\begin{align}\label{eq:zeta0}
\begin{cases}
dX^{\ld}_t=\left( -Y^{\ld}_t+B_0^{\ld}(t) \right) dt +\sigma \, dW_t, \\
dY^{\ld}_t=-\left( X^{\ld}_t + F_0^{\ld}(t) \right) dt + Z^{0,\ld}_t \, dW_t, \\
X^{\ld}_0=\xi^{\ld}, \\
Y^{\ld}_T=X^{\ld}_T +Q^{\ld}_0. 
\end{cases}
\end{align}

For any $\Theta=(X,Y,Z) \in \cM^p[0,T]:=\cM L_{\cF}^{p,c} \times \cM L_{\cF}^{p,c} \times \cM L^{p,2}_{\cF}$, we define its norm 
\begin{align*}
\lVert \Theta \rVert_{\cM^p[0,T]}^p:= \max_{\ld \in [0,1]} \E \left[\sup_{t \in [0,T]}|X^{\ld}_t|^p+\sup_{t \in [0,T]} |Y^{\ld}_t|^p+\left(\int_0^T |Z^{\ld}_t|^2 \, dt  \right)^{p/2} \right].
\end{align*}
The following lemma provides a sufficient condition for the $L^p$ boundedness of the limit  $\lim_{k \to \infty} \Theta^k$. 
\begin{lemma}\label{lem:momentbound}
Suppose $p \geq 2$ and the sequence $(\Theta^k)_{k \geq 0} \in \cM^p[0,T]$ satisfies 
\begin{align*}
&\lVert \Theta^k \rVert_{\cM^p[0,T]} \leq K, \quad k \geq 0, \\
&\lim\limits_{k \to \infty} \lVert \Theta -\Theta^k \rVert_{\cM^2[0,T]} =0.
\end{align*}
Then $\Theta \in \cM^p[0,T]$. 
\end{lemma}
\begin{proof}
See \cite[Lemma 4.1]{MR2551515}.
\end{proof}

The following lemma establishes the existence result of \eqref{eq:zeta0}. After that, we will present the main proposition of this subsection. 
\begin{lemma}\label{lem:initialize }
The FBSDE system \eqref{eq:zeta0} has a unique solution $\Theta=(X,Y,Z)$ in $\cM^p[0,T]$ and $\ld \mapsto \cL(X^{0,\ld}_t)$ is measurable for any $ t \in [0,T]$. Furthermore, we have the bound 
\begin{align}\label{eq:pthmoment}
\lVert \Theta \rVert^p_{\cM^p[0,T]} \leq K \max_{\ld \in [0,1]} \E\left[|\xi^{\ld}|^p+|Q_0^{\ld}|^p +\left(\int_0^T |B^{\ld}_0(t)|^2 \, dt  \right)^{p/2} +\left(\int_0^T |F^{\ld}_0(t)|^2 \, dt  \right)^{p/2} \right].
\end{align}
\end{lemma}
\begin{proof}
The proof of measurability is the same as Lemma~\ref{lem:picard}. Let us prove that there exits a unique solution $\Theta \in \cM^p[0,T]$. 

We can solve the system $\ld$ by $\ld$. For each $\ld \in [0,1]$, consider the BSDE 
\begin{align}\label{eq:linearBSDE}
dP^{\ld}_t&= -(-P^{\ld}_t+B^{\ld}_0(t)+F^{\ld}_0(t))\,dt + Z^{\ld}_t \, dW_t, \notag \\
P^{\ld}_T&=Q^{\ld}_0.
\end{align}
It is linear BSDE, and by \cite[Proposition 4.1.2]{MR3699487} we know that 
\begin{align*}
P^{\ld}_t= e^{-t} \E \left[ e^{-T} Q^{\ld}_0+ \int_t^T e^{-s} (B_0^{\ld}(s)+F_0^{\ld}(s)) \, ds \,  \Big| \,\cF_t\right].
\end{align*}
Therefore we get $$\E\left[ \int_0|P^{\ld}_t|^p \, dt\right] \leq K\E\left[|Q_0^{\ld}|^p +\left(\int_0^T |B^{\ld}_0(t)|^2 \, dt  \right)^{p/2} +\left(\int_0^T |F^{\ld}_0(t)|^2 \, dt  \right)^{p/2} \right].$$ Also due to \cite[Proposition 3.26]{MR1121940}, as the $p/2$-th power of quadratic variation of the martingale $$t \mapsto \E\left[ Q_0^{\ld}+\int_0^T (-P^{\ld}_s+B_0^{\ld}(s)+F_0^{\ld}(s) ) \,d s \, \Big|  \, \cF_t\right],$$ we obtain that 
\begin{align*}
\E \left[\left( \int_0^T |Z^{\ld}|^2 \, dt \right)^{p/2} \right]\leq  K\E\left[|Q_0^{\ld}|^p +\left(\int_0^T |B^{\ld}_0(t)|^2 \, dt  \right)^{p/2} +\left(\int_0^T |F^{\ld}_0(t)|^2 \, dt  \right)^{p/2} \right] .
\end{align*}
Applying BDG inequality and Gr\"{o}nwall's inequality to \eqref{eq:linearBSDE}, we can easily get that 
\begin{align*}
\E \left[ \sup_{t \in [0,T]} |P^{\ld}_t|^p \right] \leq K\E\left[|Q_0^{\ld}|^p +\left(\int_0^T |B^{\ld}_0(t)|^2 \, dt  \right)^{p/2} +\left(\int_0^T |F^{\ld}_0(t)|^2 \, dt  \right)^{p/2} \right]. 
\end{align*}

Now consider the SDE 
\begin{align*}
X_t= \xi^{\ld} + \int_0^t (-X^{\ld}_s -P^{\ld}_s+ B^{\ld}_0(s)) \, ds + \sigma W_t.
\end{align*}
Then by BDG inequality and Gr\"{o}nwall's inequality, one can easily see that 
\begin{align*}
\E\left[ \sup_{t \in [0,T]} |X_t^{\ld}|^p \right] \leq K\E\left[ |\xi^{\ld}|^p+|Q_0^{\ld}|^p +\left(\int_0^T |B^{\ld}_0(t)|^2 \, dt  \right)^{p/2} +\left(\int_0^T |F^{\ld}_0(t)|^2 \, dt  \right)^{p/2} \right].
\end{align*}
Note that $\Theta=(X^{\ld}, X^{\ld}+P^{\ld},Z^{\ld}+\sigma) $ solves \eqref{eq:zeta0}, and satisfies \eqref{eq:pthmoment}. 

\end{proof}

\begin{prop}\label{prop:2.1}
Suppose there exists a $\zeta \in [0,1]$ such that for any $B_0, F_0, Q_0 \in \cM L^{p,2}_{\cF} \times \cM L^{p,2}_{\cF} \times \cM L^p_{\mathcal{F}_T}$ there exists a unique solution $\Theta$ to \eqref{eq:conti} satisfying 
\begin{align*}
\lVert \Theta^{\zeta} \rVert^p_{\cM^p[0,T]} \leq K \max_{\ld \in [0,1]} \E \Bigg[& |\xi|^p+|Q_0^{\ld}|^p+\left(\int_0^T |B^{\ld}(t, 0)|^2+ |B_0^{\ld}(t)|^2 \right)^{p/2} \\
& +\left( \int_0^T |F^{\ld}(t,0)|^2+ |F^{\ld}_0(t)|^2 \, dt \right)^{p/2} \Bigg].
\end{align*}
Then under Assumption~\ref{assume2}, there exists an $\delta_0>0$ independent of $\zeta$ such that for any $\delta \in [0,\delta_0]$, $(B_0,F_0,Q_0) \in \cM L^{p,2}_{\cF} \times \cM L^{p,2}_{\cF} \times \cM L^p_{\cF_T}$, \eqref{eq:conti} has a unique solution $\Theta^{\zeta+\delta}=(X^{\zeta+\delta},Y^{\zeta+\delta}, Z^{\zeta+\delta})$, and the following estimate holds:
\begin{align*}
\lVert \Theta^{\zeta+\rho} \rVert^p_{\cM^p[0,T]} \leq K \max_{\ld \in [0,1]} \E \Bigg[& |\xi|^p+|Q_0^{\ld}|^p+\left(\int_0^T |B^{\ld}(t, 0)|^2+ |B_0^{\ld}(t)|^2 \right)^{p/2} \\
& +\left( \int_0^T |F^{\ld}(t,0)|^2+ |F^{\ld}_0(t)|^2 \, dt \right)^{p/2} \Bigg].
\end{align*}
\end{prop}
\begin{proof}
Denote 
\begin{align*}
B^{\zeta,\ld}(t, x, \eta, y)&=\zeta B^{\ld}(t,x ,\eta, y) -(1-\zeta)y, \\
F^{\zeta,\ld}(t, x, \eta, y)&=\zeta F^{\ld}(t,x ,\eta, y) +(1-\zeta)x,\\
Q^{\zeta, \ld}(x ,\eta)&=\zeta Q^{\ld}(x,\eta)+(1-\zeta) x. 
\end{align*}
For any pair $(x,y) \in L^2$ such that $x^{\ld}_0=\xi^{\ld}$, according to our hypothesis, there exists a unique solution $(X,Y, Z)$ to 
\begin{align*}
\begin{cases}
dX^{\ld}_t = \left( B^{\zeta,\ld}(t,X^{\ld}_t, \cL^{m}(X_t), Y^{\ld}_t)+\delta B^{\ld}(t, x^{\ld}_t, \cL^m(x_t), y^{\ld}_t)+\delta y^{\ld}_t+B_0^{\ld}(t)  \right) dt + \sigma \, dW_t, \\
dY^{\ld}_t = -\left( F^{\zeta,\ld}(t,X^{\ld}_t, \cL^{m}(X_t), Y^{\ld}_t)+\delta F^{\ld}(t, x^{\ld}_t, \cL^m(x_t), y^{\ld}_t)-\delta x^{\ld}_t+F_0^{\ld}(t)  \right) dt + Z^{\ld} \, dW_t, \\
X^{\ld}_0=\xi^{\ld}, \\
Y^{\ld}_T=Q^{\zeta,\ld}(X^{\ld}_T, \cL^m(X_T))+\delta Q^{\ld}(x^{\ld}_T, \cL^m(x_T))-\delta x_T^{\ld}+Q_0^{\ld}. 
\end{cases}
\end{align*}
Thus we have obtained maps 
\begin{align}\label{eq:contract}
\Pi: (x,y) & \mapsto (X,Y),\\
\hat{\Pi}: (x,y) & \mapsto (X,Y,Z). \notag
\end{align}
For any $(x,y) \in \cM L^{2,c}_{\mathcal{F}} \times \cM L^{2,c}_{\mathcal{F}}$, define a norm 
\begin{align*}
\lVert (x,y) \rVert^2 := \max_{\ld \in [0,1]} \left( \E[|x^{\ld}_T|^2]+\E \left[\int_0^T |x^{\ld}_t|^2+ |y^{\ld}_t|^2 \, dt \right] \right).
\end{align*}

\vspace{5pt}
\textit{Step 1: }We show that $\Pi$ is a contraction under this norm. Take $(x,y), (\wt{x}, \wt{y}) \in \cM L^{2,c}_{\mathcal{F}}$, $(X,Y)= \Pi(x,y), (\wt{X}, \wt{Y})=\Pi(\wt{x}, \wt{y})$. Denote $\Delta x= x- \wt{x}, \Delta y= y -\wt{y}, \Delta X= X- \wt{X}, \Delta Y=Y-\wt{Y}$ and $\theta^{\ld}_t=(x^{\ld}_t, \cL^m(x_t), y^{\ld}_t)$, $\wt{\theta}^{\ld}_t=(\wt{x}^{\ld}_t, \cL^m(\wt{x}_t), \wt{y}^{\ld}_t)$, $\Theta^{\ld}_t=(X^{\ld}_t, \cL^m(X_t), Y^{\ld}_t)$, $\wt{\Theta}^{\ld}_t=(\wt{X}^{\ld}_t, \cL^m(\wt{X}_t), \wt{Y}^{\ld}_t)$. Using the terminal condition of $\Delta Y^{\ld}_T$, we get that 
\begin{align*}
& \E \left[  \Delta X^{\ld}_T \Delta Y^{\ld}_T \right]= \E \left[ \Delta X^{\ld}_T \left(Q^{\zeta,\ld}(X^{\ld}_T, \cL^m(X_T))- Q^{\zeta,\ld}(\wt{X}^{\ld}_T, \cL^m({\wt{X}_T}))\right) \right] \\
&+\delta \E \left[  \Delta X^{\ld}_T \left( Q^{\ld}(x^{\ld}_T, \cL^m(x_T))- Q^{\ld}(\wt{X}^{\ld}_T, \cL^m(\wt{x}_T))\right)-\Delta X^{\ld}_T \Delta x^{\ld}_T \right].
\end{align*}
Applying It\^{o}'s formula to $ \Delta X^{\ld}_t \Delta Y^{\ld}_t$, we also have 
\begin{align*}
&\E \left[  \Delta X^{\ld}_T \Delta Y^{\ld}_T \right]= \int_0^T  \E \left[ \Delta Y^{\ld}_t \left(B^{\zeta,\ld}(t, \Theta^{\ld}_t)-B^{\zeta,\ld}(t, \wt{\Theta}^{\ld}_t) \right) \right] dt  \\
&- \int_0^T   \E \left[ \Delta X^{\ld}_t \left(F^{\zeta,\ld}(t, \Theta^{\ld}_t)-F^{\zeta,\ld}(t, \wt{\Theta}^{\ld}_t \right)\right] dt \\
&+ \int_0^T \delta  \E \left[ \Delta Y^{\ld}_t \left(B^{\ld}(t, \theta^{\ld}_t)-B^{\ld}(t, \wt{\theta}^{\ld}_t) \right)+\Delta Y^{\ld}_t \Delta y^{\ld}_t \right] dt  \\
&- \int_0^T \delta  \E \left[ \Delta X^{\ld}_t \left(F^{\ld}(t, \theta^{\ld}_t)-F^{\ld}(t, \wt{\theta}^{\ld}_t)\right) -\Delta X^{\ld}_t \Delta x^{\ld}_t \right]dt.
\end{align*} 
According to our Assumption~\ref{assume2}, we can  easily get that 
\begin{align*}
&(k-2l-\e) \left( \E \left[ | \Delta X^{ \ld }_T |^2 \right] + \E \left[\int_0^T | \Delta X^{ \ld }_t |^2+ | \Delta Y^{\ld }_t |^2 \, dt \right]\right) \\
&\leq C \delta   \left( \E \left[ | \Delta x^{ \ld }_T |^2 \right] + \E \left[\int_0^T | \Delta x^{\ld }_t |^2+ | \Delta y^{ \ld }_t |^2 \, dt \right]\right) \\ 
& + C \delta   \left(  \int_{[0,1]} \E \left[ | \Delta x^{\k }_T |^2 \right] d \k +  \int_0^T \int_{[0,1]}  \E[ | \Delta x^{\ld }_t |^2] \, d \k  \, dt \right) \\ 
&+ l \left(\int_{\k \in [0,1]} \E \left[ | \Delta X^{ \ld }_T |^2 \right] d \k +\int_0^T \int_{  [0,1]} \E \left[ | \Delta X^{ \ld }_t |^2 \right] d \k \, d t \right),
\end{align*}
where $C$ is a constant only depends on $\e$ and Lipchitz constant $l,L$. Taking maximum of both sides, one can obtain that
\begin{align*}
&(k-2l-\e) \max_{\ld \in [0,1]} \left(  \E \left[ | \Delta X^{ \ld }_T |^2 \right] +\E \left[\int_0^T | \Delta X^{ \ld }_t |^2+ | \Delta Y^{ \ld }_t |^2 \, dt \right]\right)  \\
&\leq C \delta  \max_{\ld \in [0,1]} \left(  \E \left[ | \Delta x^{\ld }_T |^2 \right] +  \E \left[\int_0^T | \Delta x^{\ld }_t |^2+ | \Delta y^{\ld }_t |^2 \, dt \right]\right) \notag \\
& \ \ \ + l  \max_{\ld \in [0,1]} \left(\E \left[ | \Delta X^{\k }_T |^2 \right] +  \int_0^T \E \left[ | \Delta X^{\k}_t |^2 \right] dt \right),\notag
\end{align*}
and hence 
\begin{align}\label{eq:conti1}
&(k-3l-\e) \max_{\ld \in [0,1]}\left(  \E \left[ | \Delta X^{ \ld }_T |^2 \right] + \E \left[\int_0^T | \Delta X^{ \ld }_t |^2+ | \Delta Y^{ \ld }_t |^2 \, dt \right]\right)  \notag \\
&\leq C \delta    \max_{\ld \in [0,1]} \left( \E \left[ | \Delta x^{\ld }_T |^2 \right] +  \E \left[\int_0^T | \Delta x^{\ld }_t |^2+ | \Delta y^{\ld }_t |^2 \, dt \right]\right).
\end{align}

First choosing $\e$ such that $k-3l > \e$, and choosing $\delta$ small enough that $k-3l -\e > C \delta $, we finished proving that $\Pi$ is a contraction. 

\vspace{5pt}
\textit{Step 2: }
Take $X^0=Y^0=0$, and define recursively $\Theta^{k+1}=(X^{k+1},Y^{k+1},Z^{k+1})= \hat{\Pi}(X^k,Y^k)$, and the limit $\Theta= (X,Y,Z)$. It is clear from our hypothesis that $\ld \mapsto (X^{k,\ld}, Y^{k,\ld})$ is measurable for any $k \in \N$, and therefore the limit $\ld \mapsto (X^{\ld}, Y^{\ld})$ is also measurable. Using $\lim\limits_{k \to \infty} \lVert (X^{k}-X, Y^{k}-Y) \rVert=0$ and some standard estimate, we obtain that 
\begin{align*}
\lim\limits_{k \to \infty} \lVert \Theta^k -\Theta \rVert_{\cM^2[0,T]} =0 .
\end{align*}

\vspace{5pt}
\textit{Step 3: } Invoking Lemma~\ref{lem:momentbound}, it remains to show that 
\begin{align*}
\lVert \Theta^{k+1} \rVert^p_{\cM^p[0,T]} \leq K \max_{\ld \in [0,1]} \E \Bigg[& |\xi|^p+|Q_0^{\ld}|^p+\left(\int_0^T |B^{\ld}(t, 0)|^2+ |B_0^{\ld}(t)|^2 \right)^{p/2} \\
& +\left( \int_0^T |F^{\ld}(t,0)|^2+ |F^{\ld}_0(t)|^2 \, dt \right)^{p/2} \Bigg].
\end{align*}
As a result of our hypothesis and the Lipschitz property of $B^{\ld}, F^{\ld}, Q^{\ld}$, we obtain that 
\begin{align*}
\lVert \Theta^{k+1} \rVert^p_{\cM^p[0,T]} \leq& K \max_{\ld \in [0,1]} \E \Bigg[ |\xi|^p+|Q_0^{\ld}|^p+\left(\int_0^T |B^{\ld}(t, 0)|^2+ |B_0^{\ld}(t)|^2 \right)^{p/2} \\
& +\left( \int_0^T |F^{\ld}(t,0)|^2+ |F^{\ld}_0(t)|^2 \, dt \right)^{p/2} \Bigg]+K\delta^p \lVert \Theta^{k} \rVert^p_{\cM^p[0,T]}.
\end{align*}
Choosing $\delta$ small enough such that $K \delta^p < 1/2$, it is then clear that for each $k \geq 1$ 
\begin{align*}
\lVert \Theta^{k} \rVert^p_{\cM^p[0,T]} \leq& 2K \max_{\ld \in [0,1]} \E \Bigg[ |\xi|^p+|Q_0^{\ld}|^p+\left(\int_0^T |B^{\ld}(t, 0)|^2+ |B_0^{\ld}(t)|^2 \right)^{p/2} \\
& +\left( \int_0^T |F^{\ld}(t,0)|^2+ |F^{\ld}_0(t)|^2 \, dt \right)^{p/2} \Bigg].
\end{align*}
Letting $k \to \infty$, we obtain the same bound for $\Theta$.
\end{proof}

\begin{thm}
Under Assumption~\ref{assume2}, there exists a unique solution to \eqref{eq:GMFG2}. 
\end{thm}
\begin{proof}
The existence can be deduced directly from Lemma~\ref{lem:initialize } and Proposition~\ref{prop:2.1}. Let us only prove the uniqueness. Suppose there are two different solutions $(X,Y,Z)$ and $(\wt{X},\wt{Y}, \wt{Z})$ to \eqref{eq:GMFG2}, and denote $\Delta X= X- \wt{X}, \Delta Y=Y-\wt{Y}$. Applying It\^{o}'s formula to $\Delta X^{\ld}_T \Delta Y^{\ld}_T$ and using similar estimation as in \textit{Step 1} of Proposition~\ref{prop:2.1}, we conclude that 
\begin{align*}
& (k-2l) \E[(\Delta X^{\ld}_T)^2]+(k-2l)  \int_0^T \E\left[  (\Delta X^{\ld}_t)^2+(\Delta Y^{\ld}_t)^2 \right] dt\\
& \leq l \int_{ [0,1]} \E[ (\Delta X^{\k}_T)^2] \, d\k  +l \int_0^T \int_{[0,1]}\E\left[ (\Delta X^{\k}_t)^2\right] d \k \, dt.
\end{align*}
Taking maximum over all $\ld \in [0,1]$, it can be readily seen that 
\begin{align*}
& (k-2l) \max_{\ld \in [0,1]}\left( \E[(\Delta X^{\ld}_T)^2]+  \int_0^T \E\left[  (\Delta X^{\ld}_t)^2+(\Delta Y^{\ld}_t)^2 \right] dt \right) \\
& \leq l \left( \int_{ [0,1]} \E[ (\Delta X^{\k}_T)^2] \, d\k  + \int_0^T \int_{[0,1]}\E\left[ (\Delta X^{\k}_t)^2\right] d \k \, dt \right),
\end{align*}
which violates our assumption $k>3l$ .
\end{proof}

\begin{remark}
The method of continuation for FBSDEs developed by \cite{Yong:1997aa} is more flexible and complicated. Here we only work under a specific assumption.
\end{remark}

\section{Stability}

Denote the solution to \eqref{eq:gGMFG} by $(x_G,y_G,z_G)$. As in \cite[Theorem 3.1]{2020arXiv200313180B}, we prove that as $\lVert G-\wt{G} \rVert_{\square} \to 0$, 
\begin{align*}
 \int_0^1  \W_{2,T}\left(\cL(x^{\ld}_G,y^{\ld}_G), \cL(x^{\ld}_{\tilde{G}},y^{\ld}_{\tilde{G}})\right) d\ld \to 0.
\end{align*}

The operator $\Gamma:=\Phi \circ \Psi $ depends on $G$, and we denote it by $\Gamma_{G}$ (see \eqref{eq:step2}, \eqref{eq:step1} for the definition of $\Phi,\Psi$). The proof stability result will be divided into three steps. 
\begin{enumerate}[(i)]
\item The operator $\Gamma_G$ is a contraction under the norm $\lVert \cdot \rVert^{I}_k$. 
\item The operator $\Gamma_G$ is continuous in $G$, i.e., as $\lVert G-\wt{G} \rVert_{\square} \to 0$, 
\begin{align*} 
\lVert \Gamma^{G}(y)- \Gamma^{\wt{G}}(y) \rVert^I_{k} \to 0.
\end{align*}
\item It holds that 
\begin{align*}
 \int_0^1  \W_{2,T}\left(\cL(x^{\ld}_G,y^{\ld}_G), \cL(x^{\ld}_{\tilde{G}},y^{\ld}_{\tilde{G}})\right) d\ld \to 0.
\end{align*}
\end{enumerate}

\subsection{Contraction mapping}

\begin{assume}\label{assume3}
(i) $B^{\ld}_G(t,x, \bs{\eta},y)= B_0(t,x ,\bs{\eta}^{\ld}, y) + \int_{[0,1]} G(\ld, \k ) \, d\k \int \hat{B}(t,x,w,y) \, \bs{\eta}^{\k}(dw) $. \\
(ii) $F^{\ld}_G(t,x, \bs{\eta},y)= F_0(t,x ,\bs{\eta}^{\ld}, y) + \int_{[0,1]} G(\ld, \k ) \, d\k \int \hat{F}(t,x,w,y) \, \bs{\eta}^{\k}(dw) $. \\
(iii) $Q^{\ld}_G(x, \bs{\eta})= Q_0(x ,\bs{\eta}^{\ld}) + \int_{[0,1]} G(\ld, \k ) \, d\k \int \hat{Q}(x,w) \, \bs{\eta}^{\k}(dw) $. \\

\end{assume}

\begin{thm}\label{thm3.1}
Suppose $(x,y,z)$ and $(\tilde{x}, \tilde{y},\tilde{z})$ are solutions of \eqref{eq:gGMFG} with graphons $G$ and $\wt{G}$ respectively. Then under Assumption~\ref{assume1} with any $p>2$ and Assumption~\ref{assume3}, we have that as $\lVert G-\wt{G} \rVert_{\square} \to 0$ and $\E\left[ \int_{[0,1]}|x^{\ld}_0 -\wt{x}^{\ld}_0|^2\, d \ld\right] \to 0$,
\begin{align}\label{eq:stabilityconclusion}
\E\left[ \int_0^1  \left(  \sup_{u \in [0,T]} |x^{\ld}_u-\wt{x}^{\ld}_u|^2 + \sup_{u \in [0,T]} |y^{\ld}_u-\wt{y}^{\ld}_u|^2+ \int_0^T |z^{\ld}_s-\wt{z}^{\ld}_s|^2 \, ds \right) \, d \ld \right]  \to 0,
\end{align}
which implies 
\begin{align}\label{eq:stabilityconclusion2}
 \int_0^1  \W_{2,T}\left(\cL(x^{\ld},y^{\ld}), \cL(\tilde{x}^{\ld},\tilde{y}^{\ld})\right) d\ld \to 0.
\end{align}
\end{thm}
\begin{proof}

\textit{Step 1:} 
By the same argument of Theorem~\ref{thm:contraction}, one can easily prove that  $\Gamma$ is a contraction with the norm $\lVert \cdot \rVert^{I,{2}}_k$ under Assumption~\ref{assume1}. For any $x \in \cM L^2_{\mathcal{F}}$, since $\lVert x \rVert_k^{{2}} \geq \lVert x \rVert^{I,{2}}_k$, the fixed point of under $\lVert \cdot \rVert_k^{{2}}$ must be the fixed under $\lVert \cdot \rVert^{I,{2}}_k$. 

\vspace{10pt}

\textit{Step 2:} Take $y$, and denote $x=\Psi_G(y), \, \wt{x}=\Psi_{\wt{G}}(y)$, $Y=\Phi_G(x), \,\wt{Y}=\Phi_{\wt{G}}(\wt{x})$. Let us calculate 
\begin{align*}
& e^{k t} |x^{\ld}_t -\wt{x}^{\ld}_t|^2 = |x^{\ld}_0 -\wt{x}^{\ld}_0|^2+ k \int_0^t e^{k s} |x^{\ld}_s -\wt{x}^{\ld}_s |^2 \, ds \notag  \\
& + 2  \int_0^t e^{k s } \left(x^{\ld}_s -\wt{x}^{\ld}_s \right) \cdot \left(B^{\ld}_{\wt{G}}(s, x^{\ld}_s, \cL^m(x_s), y^{\ld}_s )- B^{\ld}_{\wt{G}} (s, \wt{x}^{\ld}_s, \cL^m(\wt{x}_s), y^{\ld}_s) \right) ds  \notag \\
& + 2  \int_0^t e^{k s } \left(x^{\ld}_s -\wt{x}^{\ld}_s \right) \cdot \left(B^{\ld}_G(s, x^{\ld}_s, \cL^m(x_s), y^{\ld}_s )- B^{\ld}_{\wt{G}} (s, x^{\ld}_s, \cL^m(x_s), y^{\ld}_s) \right) ds \notag \\
& \leq (k-2K_1+2L_1+\e) \int_0^t e^{k s} |x^{\ld}_s -\wt{x}^{\ld}_s |^2 \, ds \notag \\
& \quad + \frac{1}{\e}   \int_0^t  e^{ks} \left(B^{\ld}_G(s, x^{\ld}_s, \cL^m(x_s), y^{\ld}_s )- B^{\ld}_{\wt{G}} (s, x^{\ld}_s, \cL^m(x_s), y^{\ld}_s) \right)^2 ds.
\end{align*}
Taking expectation and integration both sides  over $\ld$, we get that
\begin{align}\label{eq:stab1}
&\E\left[ \int_{[0,1]} e^{kT} |x^{\ld}_T-\wt{x}^{\ld}_T|^2 \, d \ld \right]+(2K_1-k-2L_1-\e)\E \left[\int_0^T  \int_{[0,1]} e^{ks} |x^{\ld}_s -\wt{x}^{\ld}_s|^2 \, d \ld  \, ds  \right] \\
& \leq \E\left[ \int_{[0,1]}|x^{\ld}_0 -\wt{x}^{\ld}_0|^2\, d \ld\right]+ C \E \left[ \int_0^T e^{ks} \, ds \int \left(B^{\ld}_G(s, x^{\ld}_s, \cL^m(x_s), y^{\ld}_s )- B^{\ld}_{\wt{G}} (s, x^{\ld}_s, \cL^m(x_s), y^{\ld}_s) \right)^2 d \ld  \right]. \notag
\end{align}

For the integrand of the last line, we show that as $\lVert G-\wt{G} \rVert_{\square} \to 0$
\begin{align}\label{eq:cutnorm1}
 \E \left[ \int_0^T e^{ks} \, ds \int \left(B^{\ld}_G(s, x^{\ld}_s, \cL^m(x_s), y^{\ld}_s )- B^{\ld}_{\wt{G}} (s, x^{\ld}_s, \cL^m(x_s), y^{\ld}_s) \right)^2 d \ld  \right]\to 0.
\end{align}
Due to Assumption~\ref{assume3}, we have that 
\begin{align*}
&\left(B^{\ld}_G(s, x^{\ld}_s, \cL^m(x_s), y^{\ld}_s )- B^{\ld}_{\wt{G}} (s, x^{\ld}_s, \cL^m(x_s), y^{\ld}_s) \right)^2 \\
& \leq \left|\int_{[0,1]} \left( G(\ld,\k) -\wt{G}(\ld,\k) \right) \, d \k  \int \hat{B}(s, x_s^{\ld}, w, y_s^{\ld})  \, \cL(x^{\k}_s)(dw) \right|^2 \\
& \leq C \left(1+|x^{\ld}_s|+ \int_{[0,1]} \sqrt{ \E[ |x_s^{\k}|^2 ]} \, d \k \right)\left|\int_{[0,1]} \left( G(\ld,\k) -\wt{G}(\ld,\k) \right) \, d \k  \int \hat{B}(s, x_s^{\ld}, w, y_s^{\ld})  \, \cL(x^{\k}_s)(dw) \right|  
\end{align*}
Taking expectation of both sides, using the boundedness of $\sup_{\ld \in [0,1]} \E[ |x_s^{\ld}|^2]$, and taking integral with respect to $\ld$, we get that 
\begin{align*}
& \E  \left[  \int \left(B^{\ld}_G(s, x^{\ld}_s, \cL^m(x_s), y^{\ld}_s )- B^{\ld}_{\wt{G}} (s, x^{\ld}_s, \cL^m(x_s), y^{\ld}_s) \right)^2  d \ld \right] \\
& \leq C \E \left[ \int  \left|\int_{[0,1]} \left( G(\ld,\k) -\wt{G}(\ld,\k) \right) \, d \k  \int \hat{B}(s, x_s^{\ld}, w, y_s^{\ld})  \, \cL(x^{\k}_s)(dw) \right| d \ld \right].
\end{align*}
By the estimation of $\mathcal{J}^{n,3}$ in the proof of \cite[Theorem 2.1]{2020arXiv200313180B} and the boundedness of $\E[|x_s^{\ld}|^{p}]+\E[|y_s^{\ld}|^{p} ]$, we obtain that as $\lVert G-\wt{G} \rVert_{\square} \to 0$
\begin{align*}
\E \left[ \int \left(B^{\ld}_G(s, x^{\ld}_s, \cL^m(x_s), y^{\ld}_s )- B^{\ld}_{\wt{G}} (s, x^{\ld}_s, \cL^m(x_s), y^{\ld}_s) \right)^2 d \ld  \right]\to 0.
\end{align*}
Then \eqref{eq:cutnorm1} follows from the fact that $t \mapsto \hat{B}(t,x,w,y)$ is Lipschitz uniformly for $(x,w,y)$.

Then let us estimate $\wt{Y}-Y$. From the equation
\begin{align*}
& e^{k T} \left|Q_G^{\ld}(x^{\ld}_T, \cL^m(x_T))- Q_{\wt{G}}^{\ld}(\wt{x}^{\ld}_T, \cL^m(\wt{x}_T))\right|^2 \\
& = e^{k t} |Y^{\ld}_t -\wt{Y}^{\ld}_t|^2 + k \int_t^T e^{k s} |Y^{\ld}_s -\wt{Y}^{\ld}_s|^2 \, ds+ \int_t^T e^{k s} |Z^{\ld}_s -\wt{Z}^{\ld}_s|^2 \, ds\\
& \ \ \  -2 \int_t^T e^{k s} \left(Y^{\ld}_s-\wt{Y}^{\ld}_s \right) \cdot \left(F^{\ld}_{\wt{G}}(s,x^{\ld}_s, \cL^m(x_s), Y^{\ld}_s)-F^{\ld}_{\wt{G}}(s, \wt{x}^{\ld}_s, \cL^m(\wt{x}_s), \wt{Y}^{\ld}_s) \right) ds \\
& \ \ \ -2 \int_t^T e^{k s} \left(Y^{\ld}_s-\wt{Y}^{\ld}_s \right) \cdot \left(F^{\ld}_G(s, x^{\ld}_s, \cL^m(x_s), Y^{\ld}_s )-F^{\ld}_{\wt{G}}(s, x^{\ld}_s, \cL^m(x_s), Y^{\ld}_s) \right) ds\\
& \ \ \ +\int_t^T e^{k s} (Y^{\ld}_s -\wt{Y}^{\ld}_s)(Z^{\ld}_s -\wt{Z}^{\ld}_s) \, dW_s^{\ld},
\end{align*}
it can be easily seen that 
\begin{align*}
& e^{k t} |Y^{\ld}_t -\wt{Y}^{\ld}_t|^2 + k \int_t^T e^{k s} |Y^{\ld}_s -\wt{Y}^{\ld}_s|^2 \, ds+ \int_t^T e^{k s} |Z^{\ld}_s -\wt{Z}^{\ld}_s|^2 \, ds \\ 
& \leq e^{k T} \left|Q_G^{\ld}(x^{\ld}_T, \cL^m(x_T))- Q_{\wt{G}}^{\ld}(\wt{x}^{\ld}_T, \cL^m(\wt{x}_T))\right|^2+(2L_2-2K_2+\e)\int_t^T e^{k s} |Y^{\ld}_s -\wt{Y}^{\ld}_s|^2 \, ds \\
& \ \ \ + L_2 \int_t^T e^{ks} |x^{\ld}_s-\wt{x}^{\ld}_s|^2 \, ds +\frac{L_2}{2}\int_t^T e^{ks} \E [ |x^{\ld}_s-\wt{x}^{\ld}_s|^2] \, ds  \\
& \ \ \ +\frac{L_2}{2} \int_t^T e^{ks} \int_{[0,1]} \E[ |x^{\k}_s-\wt{x}^{\k}_s|^2]  \, d \k  \, ds-\int_t^T e^{k s} (Y^{\ld}_s -\wt{Y}^{\ld}_s)(Z^{\ld}_s -\wt{Z}^{\ld}_s) \, dW_s^{\ld} \\
& \ \ \ +\frac{1}{\e} \int_t^T e^{ks} \left(F^{\ld}_G(s, x^{\ld}_s, \cL^m(x_s), Y^{\ld}_s )-F^{\ld}_{\wt{G}}(s, x^{\ld}_s, \cL^m(x_s), Y^{\ld}_s) \right)^2 ds.
\end{align*}
Noting that 
\begin{align*}
& e^{k T} \left|Q_G^{\ld}(x^{\ld}_T, \cL^m(x_T))- Q_{\wt{G}}^{\ld}(\wt{x}^{\ld}_T, \cL^m(\wt{x}_T))\right|^2 \\
& \leq C   \left(|x^{\ld}_T-\wt{x}^{\ld}_T|^2+\E [|x^{\ld}_T-\wt{x}^{\ld}_T|^2]+ \int_{[0,1]} \E [|x^{\beta}_T-\wt{x}^{\beta}_T|^2] \, d\beta \right) \\
& \ \ \ + C \left(Q_G^{\ld}(x^{\ld}_T, \cL^m(X_T))- Q_{\wt{G}}^{\ld}(x^{\ld}_T, \cL^m(X_T)) \right)^2,
\end{align*}
therefore one conclude that 
\begin{align}\label{eq:stab100}
& (k+2K_2-2L_2-\e)  \E \left[ \int_0^T \int_{[0,1]} e^{ks} |Y^{\ld}_s -\wt{Y}^{\ld}_s|^2  \, d \ld \, ds \right]+ \int_0^T \int_0^1 e^{k s} |Z^{\ld}_s -\wt{Z}^{\ld}_s|^2 \, d\ld \, ds \\
& \leq C \left( \E \left[ \int_{[0,1]} e^{kT} |x^{\ld}_T-\wt{x}^{\ld}_T|^2 \, d \ld  \right]  +\E \left[ \int_0^T  \int_{[0,1]} e^{ks} |x^{\ld}_s-\wt{x}^{\ld}_s|^2 \, d \ld \, ds \right] \right) \notag \\
& \ \ \ + C  \E \left[ \int_{[0,1]} \left(Q_G^{\ld}(x^{\ld}_T, \cL^m(X_T))- Q_{\wt{G}}^{\ld}(x^{\ld}_T, \cL^m(X_T)) \right)^2 d \ld\right]  \notag \\
& \ \ \ +C  \E \left[ \int_0^T e^{ks} ds \int_{[0,1]} \left(F^{\ld}_G(s, x^{\ld}_s, \cL^m(x_s), Y^{\ld}_s )-F^{\ld}_{\wt{G}}(s, x^{\ld}_s, \cL^m(x_s), Y^{\ld}_s) \right)^2  d \ld  \right]. \notag
\end{align}
Using the argument of \eqref{eq:cutnorm1}, it can be easily seen that the last two lines converge to $0$ as $\Vert G-\wt{G} \rVert_{\square} \to 0 $. In conjunction with \eqref{eq:stab1}, we finish proving that $\lVert \wt{Y} -Y \rVert^I_k \to 0$ as $\Vert \wt{G}-G \rVert_{\square} \to 0$.

\vspace{10pt}

\textit{Step 3:}
Denote by $y_G$ and $y_{\wt{G}}$ the fix point of $\Gamma_G$ and $\Gamma_{\wt{G}}$ respectively. Then it is readily seen that 
\begin{align*}
\lVert y_G-y_{\wt{G}} \rVert^I_k =&\lVert \Gamma_G(y_G)-\Gamma_{\wt{G}}(y_{\wt{G}})  \rVert^I_k \leq \lVert \Gamma_G(y_G)-\Gamma_{\wt{G}}(y_G)  \rVert^I_k+\lVert \Gamma_{\wt{G}}(y_G)-\Gamma_{\wt{G}}(y_{\wt{G}})  \rVert^I_k \\
\leq & \lVert \Gamma_G(y_G)-\Gamma_{\wt{G}}(y_G)  \rVert^I_k +\theta \lVert y_G-y_{\wt{G}} \rVert^I_k,
\end{align*}
and hence as $\lVert G-\wt{G} \rVert_{\square} \to 0$ 
\begin{align}\label{eq:stab2}
\lVert y_G-y_{\wt{G}} \rVert^I_k \leq \frac{1}{1-\theta}\lVert \Gamma_G(y_G)-\Gamma_{\wt{G}}(y_G)  \rVert^I_k \to 0.
\end{align}
Denote by $y=y_G, \, \wt{y}=y_{\wt{G}}$, $x= \Psi_G(y), \, \wt{x}=\Psi_{\wt{G}}(\wt{y})$. Similar to the derivation of \eqref{eq:stab1}, one easily obtain that 
\begin{align*}
&  \int_{[0,1]} \sup_{u \in [0,T]} \E [|x^{\ld}_u -\wt{x}^{\ld}_u|^2]  \, d\ld  \leq \E\left[ \int_{[0,1]}|x^{\ld}_0 -\wt{x}^{\ld}_0|^2\, d \ld\right] + L_1 \E \left[\int_0^T  \int_{[0,1]}  e^{ks} |y^{\ld}_s - \wt{y}^{\ld}_s|^2 \, d \ld \, d s  \right] \\
& \ \ \ + C \E \left[\int_0^T e^{ks} \, ds \int_{[0,1]}\left(B^{\ld}_G(s, x^{\ld}_s, \cL^m(x_s), y^{\ld}_s )- B^{\ld}_{\wt{G}} (s, x^{\ld}_s, \cL^m(x_s), y^{\ld}_s) \right)^2 d\ld \right].
\end{align*}
In combination with \eqref{eq:cutnorm1}, \eqref{eq:stab100}  and \eqref{eq:stab2}, it is clear that as $\lVert G-\wt{G} \rVert_{\square} \to 0$ and $\E\left[ \int_{[0,1]}|x^{\ld}_0 -\wt{x}^{\ld}_0|^2\, d \ld\right] \to 0$, 
\begin{align}\label{eq:convergence1}
 \int_0^1  \left(  \sup_{u \in [0,T]} \E[ |x^{\ld}_u-\wt{x}^{\ld}_u|^2]+ \int_0^T \E[ |y^{\ld}_u-\wt{y}^{\ld}_u|^2] \, d\ld+ \int_0^T \E [z^{\ld}_s-\wt{z}^{\ld}_s|^2 ] \, ds \right) \, d \ld   \to 0.
\end{align}

\vspace{10pt}

\textit{Step 4:} According to standard estimates, we have that 
\begin{align*}
& \sup_{u \in [0,T]} (x^{\ld}_u -\wt{x}^{\ld}_u)^2  \\
& \leq |x^{\ld}_0 -\wt{x}^{\ld}_0|^2 + C \int_0^T  |y^{\ld}_s - \wt{y}^{\ld}_s|^2+|x^{\ld}_s-\wt{x}^{\ld}_s|^2 +\E[ |x^{\ld}_s-\wt{x}^{\ld}_s|^2]\, d s  \\
& \ \ \ + C \int_0^T  \left(B^{\ld}_G(s, x^{\ld}_s, \cL^m(x_s), y^{\ld}_s )- B^{\ld}_{\wt{G}} (s, x^{\ld}_s, \cL^m(x_s), y^{\ld}_s) \right)^2 \, ds ,
\end{align*}
and 
\begin{align*}
& \sup_{u \in [0,T]} (y^{\ld}_u -\wt{y}^{\ld}_u)^2+\int_0^T (z^{\ld}_s-\wt{z}^{\ld}_s)^2 \, ds  \\
& \leq C\left( |x^{\ld}_T-\wt{x}^{\ld}_T|^2 +\E[ |x^{\ld}_T-\wt{x}^{\ld}_T|^2]\right) +\sup_{u \in [0,T]} \int_u^T (\wt{y}^{\ld}_s-{y}^{\ld}_s)(z^{\ld}_s-\wt{z}^{\ld}_s) \, dW^{\ld}_s \\
&\ \ \ + C \int_0^T  |y^{\ld}_s - \wt{y}^{\ld}_s|^2+|x^{\ld}_s-\wt{x}^{\ld}_s|^2 +\E[ |x^{\ld}_s-\wt{x}^{\ld}_s|^2]\, d s  \\
& \ \ \ + C \int_0^T  \left(F^{\ld}_G(s, x^{\ld}_s, \cL^m(x_s), y^{\ld}_s )- F^{\ld}_{\wt{G}} (s, x^{\ld}_s, \cL^m(x_s), y^{\ld}_s) \right)^2 \, ds .
\end{align*}

Taking expectation, using BDG inequality and integrating over $\ld$, we can conclude \eqref{eq:stabilityconclusion} from \eqref{eq:convergence1}.
\end{proof}

The next proposition gives a more explicit estimate than the above stability result in terms of the $L^p$ distance.
It will be used to obtain the convergence rate of propagation of chaos in the next section.

\begin{prop}\label{prop3.1}
Suppose $(x,y,z)$ and $(\tilde{x}, \tilde{y},\tilde{z})$ are solutions of \eqref{eq:gGMFG} with graphons $G$ and $\wt{G}$ respectively. Then under Assumption~\ref{assume1} with any $p \ge 2$ and Assumption~\ref{assume3}, we have that
\begin{align}\label{eq:stabilityconclusion-Lp}
& \E\left[ \int_0^1  \left(  \sup_{u \in [0,T]} |x^{\ld}_u-\wt{x}^{\ld}_u|^2 + \sup_{u \in [0,T]} |y^{\ld}_u-\wt{y}^{\ld}_u|^2+ \int_0^T |z^{\ld}_s-\wt{z}^{\ld}_s|^2 \, ds \right) \, d \ld \right] \\
& \le C\lVert G-\wt{G} \rVert_2^2 + C\int_0^1  \W_{2}^2\left(\cL(x^{\ld}_0), \cL(\tilde{x}^{\ld}_0)\right) d\ld, \notag
\end{align}
which implies 
\begin{align}\label{eq:stabilityconclusion2-Lp}
 \int_0^1  \W_{2,T}^2\left(\cL(x^{\ld},y^{\ld}), \cL(\tilde{x}^{\ld},\tilde{y}^{\ld})\right) d\ld \le C\lVert G-\wt{G} \rVert_2^2 + C\int_0^1  \W_{2}^2\left(\cL(x^{\ld}_0), \cL(\tilde{x}^{\ld}_0)\right) d\ld.
\end{align}
\end{prop}

\begin{proof}
The arguments are very similar to those in the proof of Theorem \ref{thm3.1}, except that we have explicit estimates in terms of $\lVert G-\wt{G} \rVert_2$.
So here we only highlight the differences.
In particular, in step 2, we have
\begin{align*}
	&\left(B^{\ld}_G(s, x^{\ld}_s, \cL^m(x_s), y^{\ld}_s )- B^{\ld}_{\wt{G}} (s, x^{\ld}_s, \cL^m(x_s), y^{\ld}_s) \right)^2 \\
	& \leq \left|\int_{[0,1]} \left| G(\ld,\k) -\wt{G}(\ld,\k) \right| \, d \k  \int \hat{B}(s, x_s^{\ld}, w, y_s^{\ld})  \, \cL(x^{\k}_s)(dw) \right|^2 \\
	& \leq C \left(1+|x^{\ld}_s|^2+ \int_{[0,1]} \E[| x_s^{\k}|^2] \, d \k \right)\int_{[0,1]} \left| G(\ld,\k) -\wt{G}(\ld,\k) \right|^2 d \k.
\end{align*}
Therefore the estimate \eqref{eq:cutnorm1} can be replaced by
\begin{align*}
 \E \left[ \int_0^T e^{ks} \, ds \int \left(B^{\ld}_G(s, x^{\ld}_s, \cL^m(x_s), y^{\ld}_s )- B^{\ld}_{\wt{G}} (s, x^{\ld}_s, \cL^m(x_s), y^{\ld}_s) \right)^2 d \ld  \right] \le C\lVert G-\wt{G} \rVert_2^2.
\end{align*}
Similarly, the last two terms in \eqref{eq:stab100} can be estimated by
\begin{align*}
	& \E \left[ \int_{[0,1]} \left(Q_G^{\ld}(x^{\ld}_T, \cL^m(X_T))- Q_{\wt{G}}^{\ld}(x^{\ld}_T, \cL^m(X_T)) \right)^2 d \ld\right] \\
	& \quad + \E \left[ \int_0^T e^{ks} ds \int_{[0,1]} \left(F^{\ld}_G(s, x^{\ld}_s, \cL^m(x_s), Y^{\ld}_s )-F^{\ld}_{\wt{G}}(s, x^{\ld}_s, \cL^m(x_s), Y^{\ld}_s) \right)^2  d \ld  \right] \\
	& \le C\lVert G-\wt{G} \rVert_2^2,
\end{align*}
and hence $\lVert \wt{Y} -Y \rVert^I_k \le C\lVert G-\wt{G} \rVert_2^2$. 
In step 3, we can replace \eqref{eq:stab2} by $\lVert y_G-y_{\wt{G}} \rVert^I_k \leq C\lVert G-\wt{G} \rVert_2^2$ and hence replace \eqref{eq:convergence1} by 
\begin{align*}
&\int_0^1  \left(  \sup_{u \in [0,T]} \E[ |x^{\ld}_u-\wt{x}^{\ld}_u|^2]+ \int_0^T \E[ |y^{\ld}_u-\wt{y}^{\ld}_u|^2] \, d\ld+ \int_0^T \E [z^{\ld}_s-\wt{z}^{\ld}_s|^2 ] \, ds \right) \, d \ld \\
& \leq C\lVert G-\wt{G} \rVert_2^2 + C\int_0^1  \W_{2}^2\left(\cL(x^{\ld}_0), \cL(\tilde{x}^{\ld}_0)\right) d\ld.
\end{align*}
The same argument in step 4 gives \eqref{eq:stabilityconclusion-Lp} and \eqref{eq:stabilityconclusion2-Lp}. 
\end{proof}

\subsection{Method of continuation}

\begin{thm} \label{thm3.2}
Suppose $(x,y,z)$ and $(\tilde{x}, \tilde{y},\tilde{z})$ are solutions of \eqref{eq:gGMFG} with graphons $G$ and $\wt{G}$ respectively. Then under Assumption~\ref{assume2} with any $p>2$ and Assumption~\ref{assume3}, we have the convergence \eqref{eq:stabilityconclusion} and \eqref{eq:stabilityconclusion2} as $\lVert G-\wt{G} \rVert_{\square} \to 0$ and $\E\left[ \int_{[0,1]}|x^{\ld}_0 -\wt{x}^{\ld}_0|^2\, d \ld\right] \to 0$.

\end{thm}
\begin{proof}
The proof is a mix of Proposition~\ref{prop:2.1} and Theorem~\ref{thm3.1}. For any graphon $G$, denote by $\bs{\nu}_G^{\zeta}(B_0,F_0,Q_0)$ the law of solution $X^{\zeta}$ to \eqref{eq:conti}. Let us recall the map $\Pi$ defined in \eqref{eq:contract}, and denote it by $\Pi^{\zeta}_{G}$ to indicate the dependence on the parameter $\zeta$ and coefficients $(B_G, F_G, Q_G)$. 

\vspace{10pt}

\textit{Step 1:} For any $(x,y) \in \cM L^{2,c}_{\mathcal{F}} \times \cM L^{2,c}_{\mathcal{F}}$, define a new norm 
\begin{align*}
\lVert (x,y) \rVert^{I,2}:=\int_{[0,1]} \E[|x^{\ld}_T|^2] +\int_0^T \E[ |x^{\ld}_t|^2+|y^{\ld}_t|^2 ] \, dt \, d \ld.
\end{align*}
Under Assumption~\ref{assume2}, it can be shown that $\Pi_G$ is a contraction under this norm.

\vspace{10pt}

\textit{Step 2:} Let us study $\Pi^{\zeta}_G- \Pi^{\zeta}_{\wt G}$. Take any $(x,y) \in \cM L^{2,c}_{\mathcal{F}} \times \cM L^{2,c}_{\mathcal{F}}$. Denote $(X,Y)=\Pi^{\zeta}_G(x,y)$, $(\wt X, \wt Y)= \Pi^{\zeta}_{\wt G} (x,y)$, $\Delta X= X- \wt X$, $\Delta Y= Y- \wt Y$, and $\theta^{\ld}_t=(x^{\ld}_t, \cL^m(x_t), y^{\ld}_t)$, $\Theta^{\ld}_t=(X^{\ld}_t, \cL^m(X_t), Y^{\ld}_t)$, $\wt{\Theta}^{\ld}_t=(\wt{X}^{\ld}_t, \cL^m(\wt{X}_t), \wt{Y}^{\ld}_t)$.

Let us compute 
\begin{align}\label{eq:cutnorm2}
&\E [\Delta X_T^{\ld} \Delta Y_T^{\ld}] 
 \geq   (k-2l-\e) \E[ (\Delta X^{\ld}_T)^2]- l \int_{[0,1]} \E[ (\Delta X^{\k}_T)^2] \, d \k  \\
 & - C  \E \left[ \left| Q^{\ld}_G(X^{\ld}_T, \cL^m(X_T))-Q^{\ld}_{\wt G} ( X^{\ld}_T, \cL^m(X_T)\right|^2+ \left| Q^{\ld}_G(x^{\ld}_T, \cL^m(x_T))-Q^{\ld}_{\wt G} ( x^{\ld}_T, \cL^m(x_T)\right| ^2 \right]. \notag
\end{align}
Using It\^{o}'s formula, we also obtain that 
\begin{align}\label{eq:cutnorm3}
&\E \left[  \Delta X^{\ld}_T \Delta Y^{\ld}_T \right] - \E \left[  \Delta X^{\ld}_0 \Delta Y^{\ld}_0 \right] \\
& \leq  - ( k-2l-\e) \int_0^T \E[(\Delta X^{\ld}_t)^2+ (\Delta Y^{\ld}_t )^2  ] \, dt +l \int_0^T \int_{[0,1]} \E[ (\Delta X^{\k})^2] \, d \k \, dt \notag \\
& \ \ \ +C \int_0^T \E\left[\left|B_G^{\ld}(t, \Theta^{\ld}_t)-B_{\wt G}^{\ld}(t, {\Theta}^{\ld}_t) \right|^2+\left|F_G^{\ld}(t, \Theta^{\ld}_t)-F_{\wt G}^{\ld}(t, {\Theta}^{\ld}_t) \right|^2 \right]   dt  \notag \\
&\ \ \ +C \int_0^T \E\left[\left|B_G^{\ld}(t, \theta^{\ld}_t)-B_{\wt G}^{\ld}(t, {\theta}^{\ld}_t) \right|^2+\left|F_G^{\ld}(t, \theta^{\ld}_t)-F_{\wt G}^{\ld}(t, {\theta}^{\ld}_t) \right|^2 \right]   dt . \notag
\end{align} 
Combining the above two inequalities \eqref{eq:cutnorm2}, \eqref{eq:cutnorm3} and integrating over $\ld \in [0,1]$, we get that 
\begin{align} \label{eq:cutnorm4}
& (k-3l-\e) \left( \int_{[0,1]} \E[ (\Delta X^{\ld}_T)^2] \, d \ld+\int_0^T \int_{[0,1]} \E[(\Delta X^{\ld}_t)^2+ (\Delta Y^{\ld}_t )^2  ] \, d \ld \, dt \right) \\
&  \leq C  \E \left[ \left| Q^{\ld}_G(X^{\ld}_T, \cL^m(X_T))-Q^{\ld}_{\wt G} ( X^{\ld}_T, \cL^m(X_T)\right|^2\right] \notag \\
& +C\E \left[ \left| Q^{\ld}_G(x^{\ld}_T, \cL^m(x_T))-Q^{\ld}_{\wt G} ( x^{\ld}_T, \cL^m(x_T)\right| ^2 \right] \notag \\
&+C \int_0^T \E\left[\left|B_G^{\ld}(t, \Theta^{\ld}_t)-B_{\wt G}^{\ld}(t, {\Theta}^{\ld}_t) \right|^2+\left|F_G^{\ld}(t, \Theta^{\ld}_t)-F_{\wt G}^{\ld}(t, {\Theta}^{\ld}_t) \right|^2 \right]   dt  \notag \\
&+C \int_0^T \E\left[\left|B_G^{\ld}(t, \theta^{\ld}_t)-B_{\wt G}^{\ld}(t, {\theta}^{\ld}_t) \right|^2+\left|F_G^{\ld}(t, \theta^{\ld}_t)-F_{\wt G}^{\ld}(t, {\theta}^{\ld}_t) \right|^2 \right]   dt + \int_{[0,1]}\E \left[  \Delta X^{\ld}_0 \Delta Y^{\ld}_0 \right]d\ld\notag.
\end{align}
Using the same argument as in \eqref{eq:cutnorm1}, we can show that the right hand side of the above inequality converges to $0$ as $\lVert G-\wt{G} \rVert_{\square} \to 0$ and $\E\left[ \int_{[0,1]}|X^{\ld}_0-\wt{X}^{\ld}_0|^2\, d \ld\right] \to 0$.

\vspace{10 pt}

\textit{Step 3:} Choose $\delta$ as in Proposition~\ref{prop:2.1}, $\zeta=1-\delta$,  and $(X,Y), (\wt X, \wt Y)$ to be the unique fixed point of $\Pi^{\zeta}_G, \Pi^{\zeta}_{\wt G} $ respectively. Then it is clear that 
\begin{align*}
\lVert (X,Y)-(\wt X, \wt Y) \rVert^{I,2}& =  \lVert \Pi^{\zeta}_G(X,Y)-\Pi^{\zeta}_{\wt G}(\wt X, \wt Y) \rVert^{I,2} \\
& \leq  \lVert \Pi^{\zeta}_G(X,Y)-\Pi^{\zeta}_{\wt G}(X,Y) \rVert^{I,2}+\lVert \Pi^{\zeta}_{\wt G} (X,Y)-\Pi^{\zeta}_{\wt G}(\wt X,\wt Y) \rVert^{I,2}. 
\end{align*}
Since $\Pi^{\zeta}_{\wt G}$ is $\theta$- Lipschitz with some $\theta <1$ for all graphon $\wt G$, we have that 
\begin{align*}
& \lVert (X,Y)-(\wt X, \wt Y) \rVert^{I,2} \leq \frac{1}{1-\theta}  \lVert \Pi^{\zeta}_G(X,Y)-\Pi^{\zeta}_{\wt G}(X,Y) \rVert^{I,2}.
\end{align*}
Due to \textit{Step 2}, we know that 
\begin{equation}\label{eq:thm3.2-1}
	\lVert \Pi^{\zeta}_G(X,Y)-\Pi^{\zeta}_{\wt G}(X,Y) \rVert^{I,2} \to 0
\end{equation} 
as $\lVert G -\wt{G} \rVert_{\square} \to 0$ and $\E\left[ \int_{[0,1]}|X^{\ld}_0-\wt{X}^{\ld}_0|^2\, d \ld\right] \to 0$.

\vspace{10 pt}
\textit{Step 4:}
Recall the map \eqref{eq:contract} and note that the evolution of $(X,Y), (\wt X, \wt Y)$ is given by \eqref{eq:gGMFG} with graphon $G$ and $\wt{G}$ respectively.
By It\^o's formula, we have
\begin{align*}
& \left|Q_G^{\ld}(X^{\ld}_T, \cL^m(X_T))- Q_{\wt{G}}^{\ld}(\wt{X}^{\ld}_T, \cL^m(\wt{X}_T))\right|^2 \\
& = |Y^{\ld}_t -\wt{Y}^{\ld}_t|^2 + \int_t^T |Z^{\ld}_s -\wt{Z}^{\ld}_s|^2 \, ds\\
& \ \ \  -2 \int_t^T \left(Y^{\ld}_s-\wt{Y}^{\ld}_s \right) \cdot \left(F^{\ld}_G(s, X^{\ld}_s, \cL^m(X_s), Y^{\ld}_s )-F^{\ld}_{\wt{G}}(s, \wt{X}^{\ld}_s, \cL^m(\wt{X}_s), \wt{Y}^{\ld}_s) \right) ds \\
& \ \ \ +\int_t^T (Y^{\ld}_s -\wt{Y}^{\ld}_s)(Z^{\ld}_s -\wt{Z}^{\ld}_s) \, dW_s^{\ld}.
\end{align*}
Taking expectations, integrating over $\lambda$, and using \eqref{eq:thm3.2-1}, we get
$$\int_0^1\int_0^T \E|Z^{\ld}_s -\wt{Z}^{\ld}_s|^2 \, ds \,d\lambda\to 0.$$
Lastly, using the argument of \textit{Step 4} in Theorem~\ref{thm3.1}, we can easily conclude \eqref{eq:stabilityconclusion}. 
\end{proof}

The next proposition gives a more explicit estimate than the above stability result in terms of the $L^p$ distance.
It will be used to obtain the convergence rate of propagation of chaos in the next section.

\begin{prop}\label{prop3.2}
Suppose $(x,y,z)$ and $(\tilde{x}, \tilde{y},\tilde{z})$ are solutions of \eqref{eq:gGMFG} with graphons $G$ and $\wt{G}$ respectively. Then under Assumption~\ref{assume2} with any $p \ge 2$ and Assumption~\ref{assume3}, we have the estimates \eqref{eq:stabilityconclusion-Lp} and \eqref{eq:stabilityconclusion2-Lp}.
\end{prop}

\begin{proof}
The arguments are very similar to those in the proof of Theorem \ref{thm3.2}, except that we have explicit estimates in terms of $\lVert G-\wt{G} \rVert_2$.
So here we only highlight the differences.
In particular, in step 2, from \eqref{eq:cutnorm4} we have
\begin{align*}
	& (k-3l-\e) \left( \int_{[0,1]} \E[ (\Delta X^{\ld}_T)^2] \, d \ld+\int_0^T \int_{[0,1]} \E[(\Delta X^{\ld}_t)^2+ (\Delta Y^{\ld}_t )^2  ] \, d \ld \, dt \right) \\
	& \le C\lVert G-\wt{G} \rVert_2^2 + \int_{[0,1]}\E \left[  \Delta X^{\ld}_0 \Delta Y^{\ld}_0 \right]d\ld.
\end{align*}
In step 3, we have 
\begin{align*}
\lVert (X,Y)-(\wt X, \wt Y) \rVert^{I,2} & \leq \frac{1}{1-\theta}  \lVert \Pi^{\zeta}_G(X,Y)-\Pi^{\zeta}_{\wt G}(X,Y) \rVert^{I,2} \\
& \le C\lVert G-\wt{G} \rVert_2^2 + C\int_0^1  \W_{2}^2\left(\cL(x^{\ld}_0), \cL(\tilde{x}^{\ld}_0)\right) d\ld.
\end{align*} 
Using the argument of \textit{Step 4} in Theorem \ref{thm3.2}, we have
$$\int_0^1\int_0^T \E|Z^{\ld}_s -\wt{Z}^{\ld}_s|^2 \, ds \,d\lambda \le C\lVert G-\wt{G} \rVert_2^2 + C\int_0^1  \W_{2}^2\left(\cL(x^{\ld}_0), \cL(\tilde{x}^{\ld}_0)\right) d\ld.$$ 
Using the argument of \textit{Step 4} in Theorem~\ref{thm3.1}, we have the estimates \eqref{eq:stabilityconclusion-Lp} and \eqref{eq:stabilityconclusion2-Lp}.
\end{proof}

Lastly, the following proposition shows the continuity of $\lambda \mapsto \mathcal{L}(X^\lambda,Y^\lambda)$.
The proof follows from standard coupling arguments similar to \cite[Theorem 2.1]{2020arXiv200313180B}.

\begin{prop}\label{prop-continuity}
Suppose Assumption~\ref{assume3} holds. 
Suppose either Assumption~\ref{assume1} or Assumption~\ref{assume2} holds with some $p \ge 2$.
If $G$ is (Lipschitz) continuous and $\lambda \mapsto \mathcal{L}(X^\lambda_0)$ is (Lipschitz) continuous with respect to $\W_{2}$, then $\lambda \mapsto \mathcal{L}(X^\lambda,Y^\lambda)$ is (Lipschitz) continuous with respect to $\W_{2,T}$.
\end{prop}

\section{Propagation of Chaos}

%Consider the assumption that $G_n \to G$ in $L_p$ for some $p$.
%This holds if $G$ is continuous and $G_n$ is sampled from $G$, namely $G_n(\frac{i}{n},\frac{j}{n}) = G(\frac{i}{n},\frac{j}{n})$.
%This also holds if $G_n \to G$ in the cut metric and $G \in \{0,1\}$.

%\footnote{explain why we need new assumptions on structure and continuity, as a remark}

%\footnote{explain why here we do not have law of X itself in the drift: if we insist sampling independently from the same position, then it is expected that we can get that term back.}

Consider step graphon $G_n$ such that $\|G_n-G\|_\square \to 0$ as $n \to \infty$, and
the following coupled systems of FBSDEs
\begin{align}\label{eq:sec4particle}
\begin{cases}
	dX_t^{i,n}  = B_0(t,X_t^{i,n},Y_t^{i,n})\,dt + \frac{1}{n} \sum_{j=1}^n G_n(\frac{i}{n},\frac{j}{n}) \h{B}(t,X_t^{i,n},X_t^{j,n},Y_t^{i,n})\,dt + \sigma\,dW_t^{i/n}, \\
	dY_t^{i,n}  = - F_0(t,X_t^{i,n},Y_t^{i,n})\,dt -\frac{1}{n} \sum_{j=1}^n G_n(\frac{i}{n},\frac{j}{n}) \h{F}(t,X_t^{i,n},X_t^{j,n},Y_t^{i,n})\,dt + \sum_{j=1}^n Z_t^{i,j,n}\,dW_t^{j/n}, \\
	X_0^{i,n}  = \xi^{i/n}, \\
	Y_T^{i,n}  = Q_0(X_T^{i,n}) + \frac{1}{n} \sum_{j=1}^n G_n(\frac{i}{n},\frac{j}{n}) \hat{Q}(X_T^{i,n},X_T^{j,n}),  \quad i=1, \dotso, n,
\end{cases}
\end{align}
and the following limiting system
\begin{align}\label{eq:sec4limit}
\begin{cases}
dX^{\lambda}_t = B_0(t,X^{\lambda}_t,Y^{\lambda}_t)\,dt + \int_0^1 \int_\R G(\lambda,\kappa) \h{B}\left(t,X^{\lambda}_t,x,Y^{\ld}_t \right) \cL(X_t^\kappa)(dx) \,d\kappa \, dt + \sigma \, d W^{\lambda}_t, \\
dY^{\lambda}_t = - F_0(t,X^{\lambda}_t,Y^{\lambda}_t)\,dt -\int_0^1 \int_\R G(\lambda,\kappa) \h{F}\left(t,X^{\lambda}_t,x,Y^{\ld}_t \right) \cL(X_t^\kappa)(dx) \,d\kappa \, dt +Z^{\lambda}_t \, d W^{\lambda}_t, \\
 X^{\lambda}_0 =\xi^{\lambda}, \\
 Y^{\lambda}_T = Q_0(X_T^{\lambda}) + \int_0^1 \int_{\R} G(\lambda,\kappa) \hat{Q}(X_T^{\lambda},x)\,\cL(X_T^{\kappa})(dx)\,d\kappa, \quad \ld \in [0,1].
\end{cases}
\end{align}
We will prove that solutions of \eqref{eq:sec4particle} converge to that of \eqref{eq:sec4limit}.

%
%We would like to show the following propagation of chaos 
%\begin{align}
%	& \frac{1}{n} \sum_{i=1}^n \int_0^T \E\left[ |X_t^{i,n}-X_t^{i/n}|^2 + |Y_t^{i,n}-Y_t^{i/n}|^2\right] \,dt \to 0, \label{eq:POC1} \\
%	& \int_0^T \E \left[ \W_2^2(\mu^n_t,\mu_t) \right] \,dt \to 0, \label{eq:POC2}
%\end{align}
%where $X_t^{i/n}$ is given by \eqref{eq:gGMFG}, $\mu_t^n = \frac{1}{n} \sum_{i=1}^n \delta_{X_t^{i,n}}$ and $\mu_t = \int_0^1 \cL(X_t^\lambda)\,d\lambda$.

\subsection{Contraction mapping}

The following assumption summarizes Assumptions \ref{assume1} and \ref{assume3}.

\begin{assume}\label{assume:POC-1}
(i) $B_0$ is Lipschitz in $x$, and there exists a constant $K_1 \in \R$ such that for any $(t,x,x',y) \in [0,T] \times \R^3$ 
\begin{align*}
(x-x') \cdot \left( B_0(t,x,y) -B_0(t,x',y) \right) \leq -K_1 (x-x')^2.
\end{align*}
$\hat{B}$ is $L_1$-Lipschitz in $x,x'$.\\
(ii) $F_0$ is Lipschitz in $y$, and there exists a constant $K_2 \in \R $ such that for any $(t,x,y,y') \in [0,T] \times \R^3$
\begin{align*}
(y-y') \cdot \left( F_0(t,x,y) -F_0(t,x,y') \right) \leq -K_2 (y-y')^2
\end{align*} 
$\hat{F}$ is $L_2$-Lipschitz in $x,x',y$.\\
(iii) $Q_0$, $\hat{Q}$ are Lipschitz.\\
(iv) $B_0, \hat{B}$ are bounded in $y$.\\
(v) It holds that $pK_1 +p K_2> (2p-1)L_1+(2p-2)L_2$ and there exists a constant $k \in ((2p-2)L_2-pK_2 ,p K_1-(2p-1)L_1 )$ such that 
\begin{align}\label{eq:contraction}
& (k+pK_2-(2p-2)L_2)  > \left(2^pL_1L_3^p+\frac{2^{p-1}L_1^2L_3^p+2L_1L_2}{-k+pK_1-(2p-1)L_1} \right).\end{align}
(vi) It holds that $\sup_{\ld \in [0,1]} \E[|\xi^{\ld}|^p] < +\infty$.\\
(vii) $\lambda \mapsto \mathcal{L}(\xi^\lambda)$ is continuous with respect to $\W_{2}$.
\end{assume}

We introduce the following notation that will be used in this section.
Given any measurable $y = (y^{i,n})_{i=1}^n \in \cM L^{p,c}_{\mathcal{F}}$, we define $\widehat{\Psi}_{G_n}(y):=x = (x^{i,n})_{i=1}^n$ as the unique solution to
\begin{align*}
\begin{cases}
	dx_t^{i,n}  = B_0(t,x_t^{i,n},y_t^{i,n})\,dt + \frac{1}{n} \sum_{j=1}^n G_n(\frac{i}{n},\frac{j}{n}) \h{B}(t,x_t^{i,n},x_t^{j,n},y_t^{i,n})\,dt + \sigma\,dW_t^{i/n}, \\
	x_0^{i,n}  = \xi^{i/n}, \quad i =1, \dotso,n,
\end{cases}	
\end{align*}
and define $\wt{\Psi}_{G_n}(y):=x= (x^{i,n})_{i=1}^n$ as the unique solution to
\begin{align*}
\begin{cases}
	dx_t^{i,n}  = B_0(t,x_t^{i,n},y_t^{i,n})\,dt + \frac{1}{n} \sum_{j=1}^n \int_\R G_n(\frac{i}{n},\frac{j}{n}) \h{B}(t,x_t^{i,n},x,y_t^{i,n})\cL(x_t^{j,n})(dx) \,dt + \sigma\,dW_t^{i/n}, \\
	x_0^{i,n}  = \xi^{i/n}, \quad i=1,\dotso, n.
\end{cases}
\end{align*}
For any $x = (x^{i,n})_{i=1}^n \in \cM L^{p,c}_{\mathcal{F}}$, define $\widehat{\Phi}_{G_n}(x):=y= (y^{i,n})_{i=1}^n$ to be the unique solution to backward stochastic equations
\begin{align*}
\begin{cases}
	dy_t^{i,n}  = -F_0(t,x_t^{i,n},y_t^{i,n})\,dt -\frac{1}{n} \sum_{j=1}^n G_n(\frac{i}{n},\frac{j}{n}) \h{F}(t,x_t^{i,n},x_t^{j,n},y_t^{i,n})\,dt + \sum_{j=1}^n Z_t^{i,j,n}\,dW_t^{j/n}, \\
	y_T^{i,n}  = Q_0(x_T^{i,n}) + \frac{1}{n} \sum_{j=1}^n G_n(\frac{i}{n},\frac{j}{n}) \hat{Q}(x_T^{i,n},x_T^{j,n}), \quad i=1,\dotso, n,
\end{cases}
\end{align*}
and define $\wt{\Phi}_{G_n}(x):=y= (y^{i,n})_{i=1}^n$ to be the unique solution to backward stochastic equations
\begin{align*}
\begin{cases}
	dy_t^{i,n}  = -F_0(t,x_t^{i,n},y_t^{i,n})\,dt -\frac{1}{n} \sum_{j=1}^n \int_\R G_n(\frac{i}{n},\frac{j}{n}) \h{F}(t,x_t^{i,n},x,y_t^{i,n})\cL(x_t^{j,n})(dx)\,dt + Z_t^{i,n}\,dW_t^{i/n}, \\
	y_T^{i,n}  = Q_0(x_T^{i,n}) + \frac{1}{n} \sum_{j=1}^n \int_\R G_n(\frac{i}{n},\frac{j}{n}) \hat{Q}(x_T^{i,n},x)\cL(x_t^{j,n})(dx), \quad i=1,\dotso, n.
\end{cases}
\end{align*}
We note that $\wt{\Psi}_{G_n}$ and $\wt{\Phi}_{G_n}$ are simply the maps $\Psi$ and $\Phi$ with blockwise constant graphon $G_n$ and associated piecewise constant initial states $(\xi^{\lceil n\lambda \rceil})_{\lambda \in [0,1]}$.

Let $(\wt{X}^n,\wt{Y}^n,\wt{Z}^n)$ be the unique solution of the limiting system with graphon $G_n$ and initial states $(\xi^{\lceil n\lambda \rceil})_{\lambda \in [0,1]}$.
Abusing notations, we write $\wt{X}^n = (\wt{X}^{i,n})_{i=1}^n = (\wt{X}^{n,\lambda})_{\lambda \in [0,1]}$, $\wt{Y}^n = (\wt{Y}^{i,n})_{i=1}^n = (\wt{Y}^{n,\lambda})_{\lambda \in [0,1]}$ and $\wt{Z}^n = (\wt{Z}^{i,n})_{i=1}^n = (\wt{Z}^{n,\lambda})_{\lambda \in [0,1]}$.
Note that $Y^n$ and $\wt{Y}^n$ are the fix point of $\widehat{\Gamma}_{G_n}:=\widehat{\Phi}_{G_n}\circ\widehat{\Psi}_{G_n}$ and $\widetilde{\Gamma}_{G_n}:=\wt{\Phi}_{G_n}\circ\wt{\Psi}_{G_n}$ respectively.

\begin{thm} \label{thm4.1}
	
	Suppose Assumption~\ref{assume:POC-1} holds and $\|G_n - G\|_\square \to 0$. 
	Then
	\begin{align}\label{eq:POC-0}
		& \E\left[ \int_0^1  \left(  \sup_{t \in [0,T]} |X_t^{\lceil n\lambda \rceil,n}-X_t^{\lambda}|^2 + \sup_{t \in [0,T]} |Y_t^{\lceil n\lambda \rceil,n}-Y_t^{\lambda}|^2 \, ds \right. \right. \\
		& \qquad \left. \left. + \int_0^T |Z^{\lceil n\lambda \rceil,\lceil n\lambda \rceil,n}_t -Z^{\lambda}_t|^2\,dt \right) \, d \ld \right] \to 0. \notag
	\end{align}	
	If in addition $G$ is continuous, then
	\begin{align}
		\frac{1}{n} \sum_{i=1}^n \E \sup_{0 \le t \le T} \left[ |X_t^{i,n}-X_t^{i/n}|^2 + |Y_t^{i,n}-Y_t^{i/n}|^2\right] \,dt & \to 0, \label{eq:POC-1} \\
		\sup_{t \in [0,T]} \E \left[ \W_2^2(\nu^n_t,\nu_t) \right] \,dt & \to 0, \label{eq:POC-2}
	\end{align}		
	where $\nu_t^n = \frac{1}{n} \sum_{i=1}^n \delta_{(X_t^{i,n},Y_t^{i,n})}$ and $\nu_t = \int_0^1 \cL(X_t^\lambda,Y_t^\lambda)\,d\lambda$.
\end{thm}

\begin{proof}
%	Let $\widehat{\Psi}_{G_n} \colon y \mapsto x$ and $\widehat{\Phi}_{G_n} \colon x \mapsto y$ be the map for the $n$-particle system.
%	
	\textit{Step 1:}
	Take $y=(y_{\frac{k}{n}}:k=1,\dotsc,n)$, and denote 
	$x^n=\widehat{\Psi}_{G_n}(y)$, 
	$\tilde{x}^n=\wt{\Psi}_{G_n}(y)$, 
	$y^n=\widehat{\Phi}_{G_n}(x^n)$, 
	$\tilde{y}^n=\wt{\Phi}_{G_n}(\tilde{x}^n)$.
	By It\^o's formula,  
	\begin{align*}
	& e^{k t} |x_t^{i,n}-\tilde{x}_t^{i,n}|^2 =k \int_0^t e^{k s} |x_s^{i,n}-\tilde{x}_s^{i,n}|^2 \, ds \\
	& + 2  \int_0^t e^{k s } \left(x_s^{i,n}-\tilde{x}_s^{i,n}\right) \cdot \left(B_0(s,x_s^{i,n},y_s^{i,n}) - B_0(s,\tilde{x}_s^{i,n},y_s^{i,n})\right) ds \\
	& + 2  \int_0^t e^{k s } \left(x_s^{i,n}-\tilde{x}_s^{i,n}\right) \cdot \frac{1}{n} \sum_{j=1}^n G_n(\frac{i}{n},\frac{j}{n}) \left( \h{B}(s,x_s^{i,n},x_s^{j,n},y_s^{i,n}) - \h{B}(s,\tilde{x}_s^{i,n},\tilde{x}_s^{j,n},y_s^{i,n}) \right) ds \\
	& + 2  \int_0^t e^{k s } \left(x_s^{i,n}-\tilde{x}_s^{i,n}\right) \cdot \frac{1}{n} \sum_{j=1}^n G_n(\frac{i}{n},\frac{j}{n}) \left( \h{B}(s,\tilde{x}_s^{i,n},\tilde{x}_s^{j,n},y_s^{i,n}) - \int_\R \h{B}\left(s,\tilde{x}_s^{i,n},x,y_s^{i,n} \right) \cL(\tilde{x}_s^{j,n})(dx) \right) ds \\
	& \leq (k-2K_1+3L_1+\e) \int_0^t e^{k s} |x^{i,n}_s -\wt{x}^{i,n}_s |^2 \, ds + L_1\frac{1}{n} \sum_{j=1}^n \int_0^t e^{k s} |x^{j,n}_s -\wt{x}^{j,n}_s |^2 \, ds \\
	& \quad + \frac{1}{\e}   \int_0^t  e^{ks} \left( \frac{1}{n} \sum_{j=1}^n G_n(\frac{i}{n},\frac{j}{n}) \left( \h{B}(s,\tilde{x}_s^{i,n},\tilde{x}_s^{j,n},y_s^{i,n}) - \int_\R \h{B}\left(s,\tilde{x}_s^{i,n},x,y_s^{i,n} \right) \cL(\tilde{x}_s^{j,n})(dx) \right) \right)^2 ds.
	\end{align*}
	Taking expectations and the average over $i$, we have
	\begin{align*}
		& \frac{1}{n}\sum_{i=1}^n \E\left[e^{k t} |x_t^{i,n}-\tilde{x}_t^{i,n}|^2\right] \\
		& \le (k-2K_1+4L_1+\e) \int_0^t \frac{1}{n}\sum_{i=1}^n \E\left[e^{k s} |x^{i,n}_s -\wt{x}^{i,n}_s |^2 \right] ds \\
		& + \frac{1}{\e} \int_0^t  e^{ks} \frac{1}{n}\sum_{i=1}^n \E \left( \frac{1}{n} \sum_{j=1}^n G_n(\frac{i}{n},\frac{j}{n}) \left( \h{B}(s,\tilde{x}_s^{i,n},\tilde{x}_s^{j,n},y_s^{i,n}) - \int_\R \h{B}\left(s,\tilde{x}_s^{i,n},x,y_s^{i,n} \right) \cL(\tilde{x}_s^{j,n})(dx) \right) \right)^2 ds.
	\end{align*}
	For the last line, we have the estimation
	\begin{align*}
	&\frac{1}{\e} \int_0^t  e^{ks} \E \left( \frac{1}{n} \sum_{j=1}^n G_n(\frac{i}{n},\frac{j}{n}) \left( \h{B}(s,\tilde{x}_s^{i,n},\tilde{x}_s^{j,n},y_s^{i,n}) - \int_\R \h{B}\left(s,\tilde{x}_s^{i,n},x,y_s^{i,n} \right) \cL(\tilde{x}_s^{j,n})(dx) \right) \right)^2 ds \\
	& = \frac{1}{\e} \int_0^t  e^{ks} \frac{1}{n^2} \sum_{j=1}^n G_n^2(\frac{i}{n},\frac{j}{n}) \E \left( \h{B}(s,\tilde{x}_s^{i,n},\tilde{x}_s^{j,n},y_s^{i,n}) - \int_\R \h{B}\left(s,\tilde{x}_s^{i,n},x,y_s^{i,n} \right) \cL(\tilde{x}_s^{j,n})(dx) \right)^2 ds \\
	&\leq \frac{C}{n},
	\end{align*}
	due to the boundedness of $\E[\sup_{s \in [0,T]}|\tilde{x}_s^{i,n}|^2]$ and Lipschitz property of $\hat{B}$. 
	Therefore
	\begin{equation}\label{eq:poc1}
		\frac{1}{n}\sum_{i=1}^n \sup_{t \in [0,T]} \E\left[e^{k t} |x_t^{i,n}-\tilde{x}_t^{i,n}|^2\right] \le \frac{C}{n}.
	\end{equation}
%	{\color{red} Therefore it holds that 
%	\begin{align}
%	&\E\left[ \int_{[0,1]} e^{kT} |x^{\ld}_T-\wt{x}^{\ld}_T|^2 \, d \ld \right]+(2K_1-k-2L_1-\e)\E \left[\int_0^T  \int_{[0,1]} e^{ks} |x^{\ld}_s -\wt{x}^{\ld}_s|^2 \, d \ld  \, ds  \right] \\
%	& \leq C \int_0^T \int_{[0,1]} \int_{[0,1]} e^{ks} |\wt{G}(\ld, \k)-G(\ld, \k)|\left(1+\E[|x^{\ld}_s|^2]+\E[|y^{\ld}_s|^2]+\int_{[0,1]} \E[|x^{\beta}_s|^2] \, d \beta \right) d \ld \, d \k \, ds. \notag
%	\end{align}
%	Due to the boundedness of $(\E[|x^{\ld}_s|^2],\E[|y^{\ld}_s|^2])$ in $\ld$, the right hand side converges to $0$ as $\lVert \wt{G}-G \rVert_{\square} \to 0$.} 
	
	\textit{Step 2:}
	Then let us estimate $y^n-\wt{y}^n$. From the equation
	\begin{align*}
	& e^{k T} \E\left|y_T^{i,n}-\wt{y}_T^{i,n}\right|^2 \\
	& = e^{k t} \E|y^{i,n}_t -\wt{y}^{i,n}_t|^2 + k \E\int_t^T e^{k s} |y^{i,n}_s -\wt{y}^{i,n}_s|^2 \, ds+ \sum_{j=1}^n \E\int_t^T e^{k s} |Z^{i,j,n}_s -\delta_{ij}\wt{Z}^{i,n}_s|^2 \, ds\\
	&  -2 \E\int_t^T e^{k s} \left(y^{i,n}_s-\wt{y}^{i,n}_s \right) \cdot \left(F_0(s,x^{i,n}_s,y^{i,n}_s)-F_0(s,\wt{x}^{i,n}_s,\wt{y}^{i,n}_s) \right) ds \\
	& -2 \E\int_t^T e^{k s} \left(y^{i,n}_s-\wt{y}^{i,n}_s \right) \cdot \frac{1}{n} \sum_{j=1}^n G_n(\frac{i}{n},\frac{j}{n}) \left(\hat{F}(s, x^{i,n}_s, x^{j,n}_s, y^{i,n}_s) - \hat{F}(s, \wt{x}^{i,n}_s, \wt{x}^{j,n}_s, \wt{y}^{i,n}_s) \right) ds \\
	& -2 \E\int_t^T e^{k s} \left(y^{i,n}_s-\wt{y}^{i,n}_s \right) \cdot \frac{1}{n} \sum_{j=1}^n G_n(\frac{i}{n},\frac{j}{n}) \left(\hat{F}(s, \wt{x}^{i,n}_s, \wt{x}^{j,n}_s, \wt{y}^{i,n}_s) - \int_\R \hat{F}(s, \wt{x}^{i,n}_s, x, \wt{y}^{i,n}_s)\,\cL(\tilde{x}_s^{j,n})(dx) \right) ds,
	\end{align*}
	it can be easily seen that 
	\begin{align} \label{eq:poc-pf1}
	& e^{k t} \E|y^{i,n}_t -\wt{y}^{i,n}_t|^2 + k \E\int_t^T e^{k s} |y^{i,n}_s -\wt{y}^{i,n}_s|^2 \, ds+ \sum_{j=1}^n \E\int_t^T e^{k s} |Z^{i,j,n}_s -\delta_{ij}\wt{Z}^{i,n}_s|^2 \, ds \\ 
	& \leq e^{k T} \E\left|y_T^{i,n}-\wt{y}_T^{i,n}\right|^2 + (4L_2-2K_2+\e)\E\int_t^T e^{k s} |y^{i,n}_s -\wt{y}^{i,n}_s|^2 \, ds \notag \\
	& \ \ \ + L_2 \E\int_t^T e^{ks} |x^{i,n}_s-\wt{x}^{i,n}_s|^2 \, ds +L_2 \frac{1}{n}\sum_{j=1}^n\E\int_t^T e^{ks} |x^{j,n}_s-\wt{x}^{j,n}_s|^2 \, ds \notag \\
	& \ \ \ +\frac{1}{\e} \E\int_t^T e^{ks} \left( \frac{1}{n} \sum_{j=1}^n G_n(\frac{i}{n},\frac{j}{n}) \left(\hat{F}(s, \wt{x}^{i,n}_s, \wt{x}^{j,n}_s, \wt{y}^{i,n}_s) - \int_\R \hat{F}(s, \wt{x}^{i,n}_s, x, \wt{y}^{i,n}_s)\,\cL(\tilde{x}_s^{j,n})(dx) \right) \right)^2 ds. \notag
	\end{align}
	Noting that 
	\begin{equation*}
		\frac{1}{\e} \E\int_t^T e^{ks} \left( \frac{1}{n} \sum_{j=1}^n G_n(\frac{i}{n},\frac{j}{n}) \left(\hat{F}(s, \wt{x}^{i,n}_s, \wt{x}^{j,n}_s, \wt{y}^{i,n}_s) - \int_\R \hat{F}(s, \wt{x}^{i,n}_s, x, \wt{y}^{i,n}_s)\,\cL(\tilde{x}_s^{j,n})(dx) \right) \right)^2 ds \le \frac{C}{n},
	\end{equation*}
	and also
	\begin{align*}
	& e^{k T} \E\left|y_T^{i,n}-\wt{y}_T^{i,n}\right|^2 \leq C   \E\left(|x^{i,n}_T-\wt{x}^{i,n}_T|^2 + \frac{1}{n}\sum_{j=1}^n|x^{i,n}_T-\wt{x}^{i,n}_T|^2 + \frac{1}{n} \right).
	\end{align*}
	Therefore we conclude from \eqref{eq:poc1} that 
	\begin{align*}
	& (k+2K_2-4L_2-\e)  \E \left[ \int_0^T \frac{1}{n} \sum_{i=1}^n e^{ks} |y^{i,n}_s -\wt{y}^{i,n}_s|^2  \, ds \right] \leq \frac{C}{n}.
	\end{align*}

	\textit{Step 3:}
	Recall the processes $Y^n$, $\wt{Y}^n$ and $\wt{Z}^n$.
	Then it is readily seen that 
	\begin{align*}
	\lVert Y^n-\wt{Y}^n \rVert^I_k =&\lVert \hat{\Gamma}_{G_n}(Y^n)-\tilde{\Gamma}_{G_n}(\wt{Y}^n)  \rVert^I_k \leq \lVert \hat{\Gamma}_{G_n}(Y^n)-\tilde{\Gamma}_{G_n}(Y^n)  \rVert^I_k +\lVert \tilde{\Gamma}_{G_n}(Y^n)-\tilde{\Gamma}_{G_n}(\wt{Y}^n)  \rVert^I_k \\
	\leq & \lVert \hat{\Gamma}_{G_n}(Y^n)-\tilde{\Gamma}_{G_n}(Y^n)  \rVert^I_k +\theta \lVert Y^n-\wt{Y}^n \rVert^I_k,
	\end{align*}
	and hence as $n \to \infty$,
	\begin{align*}
%	\label{eq:poc2}
	\lVert Y^n-\wt{Y}^n \rVert^I_k \leq \frac{1}{1-\theta}\lVert \hat{\Gamma}_{G_n}(Y^n)-\tilde{\Gamma}_{G_n}(Y^n)  \rVert^I_k \le \frac{C}{n} \to 0.
	\end{align*}
	Similar to the derivation of \eqref{eq:poc1}, one can easily obtain that
	\begin{equation*}
		\frac{1}{n}\sum_{i=1}^n \sup_{t \in [0,T]} \E\left[ e^{k t} |X_t^{i,n}-\wt{X}_t^{i,n}|^2 \right] \le C\lVert Y^n-\wt{Y}^n \rVert^I_k + \frac{C}{n} \to 0.
	\end{equation*}	
	Combining these with \eqref{eq:poc-pf1} we have
	\begin{align*}
		& \frac{1}{n}\sum_{i=1}^n \left( \sup_{t \in [0,T]} \E |X_t^{i,n}-\wt{X}_t^{i,n}|^2 + \int_0^T \E |Y_t^{i,n}-\wt{Y}_t^{i,n}|^2\,dt + \sum_{j=1}^n \int_0^T \E |Z^{i,j,n}_t -\delta_{ij}\wt{Z}^{i,n}_t|^2\,dt \right) \\
		& \le \frac{C}{n} \to 0.
	\end{align*}	
	
	\textit{Step 4:} 
	Similar to the arguments of \textit{Step 4} in Theorem \ref{thm3.1}, we have
	\begin{align} \label{eq:POC-pf2}
		& \E \left[ \frac{1}{n}\sum_{i=1}^n \left( \sup_{t \in [0,T]} |X_t^{i,n}-\wt{X}_t^{i,n}|^2 + \sup_{t \in [0,T]} |Y_t^{i,n}-\wt{Y}_t^{i,n}|^2 + \sum_{j=1}^n \int_0^T |Z^{i,j,n}_t -\delta_{ij}\wt{Z}^{i,n}_t|^2\,dt \right) \right] \\
		& \le \frac{C}{n} \to 0. \notag
	\end{align}		
	Using Theorem \ref{thm3.1} and the assumptions that $\|G_n - G\|_\square \to 0$ and $\lambda \mapsto \mathcal{L}(\xi^\lambda)$ is continuous with respect to $\W_{2}$, we have
	\begin{equation*}
	\E\left[ \int_0^1  \left(  \sup_{u \in [0,T]} |X^{\ld}_u-\wt{X}^{n,\ld}_u|^2 + \sup_{u \in [0,T]} |Y^{\ld}_u-\wt{Y}^{n,\ld}_u|^2+ \int_0^T |Z^{\ld}_s-\wt{Z}^{n,\ld}_s|^2 \, ds \right) \, d \ld \right] \to 0.
	\end{equation*}	
	Combining the last two displays gives \eqref{eq:POC-0}.
	
	Finally we will show \eqref{eq:POC-1} and \eqref{eq:POC-2} under the assumption that $G$ is continuous.
	Using Proposition \ref{prop-continuity}, we have the following convergence for the limiting system
	\begin{equation*}
		\E\left[ \int_0^1  \left(  \sup_{t \in [0,T]} |X_t^{\lceil n\lambda \rceil}-X_t^{\lambda}|^2 + \sup_{t \in [0,T]} |Y_t^{\lceil n\lambda \rceil}-Y_t^{\lambda}|^2 \right) \, d \ld \right]  \to 0.
	\end{equation*}		
	Combining this with \eqref{eq:POC-0} gives \eqref{eq:POC-1}.	
	By \eqref{eq:POC-1} we have
	\begin{equation*}
		\sup_{0 \le t \le T} \E\left[ \W_2^2(\nu_t^n,\bar\nu_t^n) \right] \to 0,
	\end{equation*}
	where $\bar\nu_t^n = \frac{1}{n} \sum_{i=1}^n \delta_{(X_t^{i/n},Y_t^{i/n})}$.
	Using the independence and moment bound on $(X_t^\lambda,Y_t^\lambda)$ (see e.g.\ \cite[Lemma A.1]{2020arXiv200810173B}), we have
	\begin{equation*}
		\sup_{0 \le t \le T} \E\left[ \W_2^2(\bar\nu_t^n,\E\bar\nu_t^n) \right] \to 0.
	\end{equation*}
	From Proposition \ref{prop-continuity} we have
	\begin{equation*}
		\sup_{0 \le t \le T} \E\left[ \W_2^2(\E\bar\nu_t^n,\nu_t) \right] \to 0.
	\end{equation*}
	Combining these three displays gives \eqref{eq:POC-2}.
%	Denote by $y=y_G, \, \wt{y}=y_{\wt{G}}$, $x= \Psi_G(y), \, \wt{x}=\Psi_{\wt{G}}(\wt{y})$. As in the derivation of \eqref{eq:stab1}, one easily obtain that 
%	\begin{align*}
%	& \int_{[0,1]} \sup_{u \in [0,T]} \W_2\left(\bs{\nu}_{{G}}^{\ld}(u), \bs{\nu}_{\wt{G}}^{\ld}(u)\right)^2 d\ld \leq \int_{[0,1]} \sup_{u \in [0,T]} \E [|x^{\ld}_u -\wt{x}^{\ld}_u|^2]  \, d\ld \\
%	& \leq L_1 \E \left[ \int_{[0,1]} \int_0^T  e^{ks} |y^{\ld}_s - \wt{y}^{\ld}_s|^2 \, d s \, d \ld \right] \\
%	& \ \ \ + C \int_0^T \int_{[0,1]} \int_{[0,1]} e^{ks} |\wt{G}(\ld, \k)-G(\ld, \k)|\left(1+\E[|x^{\ld}_s|^2]+\E[|y^{\ld}_s|^2]+\int_{[0,1]} \E[|x^{\beta}_s|^2] \, d \beta \right) d \ld \, d \k \, ds.
%	\end{align*}
%	In combination with \eqref{eq:stab2}, it is clear that as $\lVert G-\wt{G} \rVert_{\square} \to 0$, 
%	\begin{align*}
%	\int_{[0,1]}  \W_{2,T}\left(\bs{\nu}_{{G}}^{\ld}, \bs{\nu}_{\wt{G}}^{\ld}\right) d\ld \to 0.
%	\end{align*}
	
\end{proof}

Under certain assumptions we can obtain the rate of convergence.

\begin{prop}\label{prop4.1}
	Suppose Assumption~\ref{assume:POC-1} holds.
%	 and $\lambda \mapsto \mathcal{L}(\xi^\lambda)$ is continuous with respect to $\W_{2}$. 
	Then
	\begin{align*} 
%	\label{eq:rate1}
		& \E\left[ \int_0^1  \left(  \sup_{t \in [0,T]} |X_t^{\lceil n\lambda \rceil,n}-X_t^{\lambda}|^2 + \sup_{t \in [0,T]} |Y_t^{\lceil n\lambda \rceil,n}-Y_t^{\lambda}|^2 \, ds + \int_0^T |Z^{\lceil n\lambda \rceil,\lceil n\lambda \rceil,n}_t -Z^{\lambda}_t|^2\,dt \right) \, d \ld \right] \\
		& \le \frac{C}{n} + C\|G_n - G\|_2^2 + C\int_0^1  \W_{2}^2\left(\cL(X^{\ld}_0), \cL(X^{\lceil n\lambda \rceil,n}_0)\right) d\ld.
	\end{align*}
\end{prop}

\begin{proof}
	From Proposition \ref{prop3.1} we have
	\begin{align*}
	& \E\left[ \int_0^1  \left(  \sup_{u \in [0,T]} |X^{\ld}_u-\wt{X}^{n,\ld}_u|^2 + \sup_{u \in [0,T]} |Y^{\ld}_u-\wt{Y}^{n,\ld}_u|^2+ \int_0^T |Z^{\ld}_s-\wt{Z}^{n,\ld}_s|^2 \, ds \right) \, d \ld \right] \\
	& \le C\|G_n - G\|_2^2 + C\int_0^1  \W_{2}^2\left(\cL(X^{\ld}_0), \cL(\tilde{X}^{n,\ld}_0)\right) d\ld.
	\end{align*}	
	The result then follows by combining this with \eqref{eq:POC-pf2} and the observation that $\cL(\wt{X}^{n,\ld}_0)=\cL(X^{\lceil n\lambda \rceil,n}_0)$.
\end{proof}

\begin{remark} \label{rmk-4.1}
	Proposition \ref{prop4.1} provides a rate of convergence in terms of the $L^2$ convergence $\|G_n - G\|_2 \to 0$. 
	Such a convergence holds, for example, if $\|G_n - G\|_\square \to 0$ and $G \in \{0,1\}$ (see e.g. \cite[Proposition 8.24]{lovasz2012large}).
	It also holds (by dominated convergence theorem) if $G$ is continuous and $G_n$ is sampled from $G$, namely $G_n(\frac{i}{n},\frac{j}{n}) := G(\frac{i}{n},\frac{j}{n})$.
\end{remark}

The following is a more precise rate of convergence under Lipschitz conditions.
 
\begin{corollary}\label{cor:4.1}
	Suppose Assumption~\ref{assume:POC-1} holds.
%	 and $\lambda \mapsto \mathcal{L}(\xi^\lambda)$ is continuous with respect to $\W_{2}$. 
	Suppose $G$ is Lipschitz continuous, $\lambda \mapsto \mathcal{L}(\xi^\lambda)$ is Lipschitz continuous with respect to $\W_{2}$, and $G_n$ is sampled from $G$, namely $G_n(\frac{i}{n},\frac{j}{n}) = G(\frac{i}{n},\frac{j}{n})$.
	Then $\|G_n - G\|_2 \le \frac{C}{n}$ and hence
	\begin{align} 
		& \E\left[ \int_0^1  \left(  \sup_{t \in [0,T]} |X_t^{\lceil n\lambda \rceil,n}-X_t^{\lambda}|^2 + \sup_{t \in [0,T]} |Y_t^{\lceil n\lambda \rceil,n}-Y_t^{\lambda}|^2 \, ds \right. \right. \label{eq:rate-0} \\
		& \qquad \left. \left. + \int_0^T |Z^{\lceil n\lambda \rceil,\lceil n\lambda \rceil,n}_t -Z^{\lambda}_t|^2\,dt \right) \, d \ld \right] \le \frac{C}{n}, \notag \\
		& \frac{1}{n} \sum_{i=1}^n \E \sup_{0 \le t \le T} \left[ |X_t^{i,n}-X_t^{i/n}|^2 + |Y_t^{i,n}-Y_t^{i/n}|^2\right] \,dt \le \frac{C}{n}, \label{eq:rate-1} \\
		& \sup_{t \in [0,T]} \E \left[ \W_2^2(\nu^n_t,\nu_t) \right] \,dt \le C(n^{-1/2}+n^{-(p-2)/p}). \label{eq:rate-2} 
	\end{align} 
\end{corollary}

\begin{proof}
The estimate $\|G_n - G\|_2 \le \frac{C}{n}$ follows from the Lipschitz continuity of $G$.
Applying this to Proposition \ref{prop4.1} gives \eqref{eq:rate-0}.
Combining \eqref{eq:rate-0} and Proposition \ref{prop-continuity}, we have \eqref{eq:rate-1}.
From \eqref{eq:rate-1} we have
\begin{equation*}
	\sup_{0 \le t \le T} \E\left[ \W_2^2(\nu_t^n,\bar\nu_t^n) \right] \le \frac{C}{n}.
\end{equation*}
Using the independence and moment bound on $(X_t^\lambda,Y_t^\lambda)$ (see e.g.\ \cite[Lemma A.1]{2020arXiv200810173B}), we have
\begin{equation*}
	\sup_{0 \le t \le T} \E\left[ \W_2^2(\bar\nu_t^n,\E\bar\nu_t^n) \right] \le C(n^{-1/2}+n^{-(p-2)/p}).
\end{equation*}
From Proposition \ref{prop-continuity} we have
\begin{equation*}
	\sup_{0 \le t \le T} \E\left[ \W_2^2(\E\bar\nu_t^n,\nu_t) \right] \le \frac{C}{n^2}.
\end{equation*}
Combining these three displays gives \eqref{eq:rate-2}.
\end{proof}

\begin{remark}\label{rmk-4.2}
	Although we work for $\R$-valued stochastic process $X^\lambda$, similar arguments can be used to show that Theorem \ref{thm4.1}, Proposition \ref{prop4.1} and Corollary \ref{cor:4.1} also hold for $\R^d$-valued setup.
	In that case, the rate in \eqref{eq:rate-2} will be slightly different (see e.g.\ \cite[Lemma A.1]{2020arXiv200810173B})
\end{remark}

%Here we need more continuity results. Ref: \cite{2020arXiv200408351L}.

%\begin{lemma}
%	The convergence \eqref{eq:POC1} holds, namely
%	\begin{equation*}
%		\frac{1}{n} \sum_{i=1}^n \E\left[ \sup_{t \in [0,T]} |X_t^{i,n}-X_t^{i/n}|^2 \right] \to 0
%	\end{equation*}
%\end{lemma}
%
%\begin{proof}
%	Applying It\^o's formula to $e^{kt} |Y_t^{i,n}-Y_t^{i/n}|^2$ for some $k \ge 0$, we have
%	\begin{align*}
%		e^{kt} |Y_t^{i,n}-Y_t^{i/n}|^2 & = e^{kT} |Y_T^{i,n}-Y_T^{i/n}|^2 + 2\sum_{j=1}^n \int_t^T 
%	\end{align*}
%\end{proof}

\subsection{Method of continuation}

\begin{assume}\label{assume:POC-2}
	(i) $B_0,\hat{B},F_0,\hat{F},Q_0,\hat{Q}$ are $l$-Lipschitz. \\
	(ii) There exist a positive constant $k>3l$ such that 
	\begin{align*}
	-\Delta x \left(F_0(t,\theta)-F_0(t,\wt{\theta}) \right) +\Delta y\left( B_0(t,\theta)-B_0(t,\wt{\theta})  \right) 
	& \leq -k   (\Delta x)^2-k(\Delta y)^2 , \\
	\Delta x\left( Q_0(x)-Q_0(\wt{x}) \right) 
	& \geq k (\Delta x)^2,
	\end{align*}
	where $\Delta x:= x-\wt{x}$, $\Delta y:= y-\wt{y}$, $\theta=(x,y)$,$\wt{\theta}=(\wt{x},\wt{y})$. \\
	(iii) It holds that $\sup_{\ld \in [0,1]} \E[|\xi|^p] < +\infty$, and $(B(\cdot, 0), F(\cdot, 0), Q) \in \cM L^{p,2}_{\cF} \times \cM L^{p,2}_{\cF} \times \cM L^p_{\cF_T}$. \\
	(iv) $\lambda \mapsto \mathcal{L}(\xi^\lambda)$ is continuous with respect to $\W_{2}$.
\end{assume}

\begin{thm}\label{thm4.2}
	
	Suppose Assumption~\ref{assume:POC-2} holds and $\|G_n - G\|_\square \to 0$. 
	Then \eqref{eq:POC-0} holds.
	If in addition $G$ is continuous, then \eqref{eq:POC-1} and \eqref{eq:POC-2} hold.
\end{thm}

\begin{proof}	
	We will use the same notation as above Theorem \ref{thm4.1}.
	That is, let $(\wt{X}^n,\wt{Y}^n,\wt{Z}^n)$ be the unique solution of the limiting system with graphon $G_n$ and initial states $(\xi^{\lceil n\lambda \rceil})_{\lambda \in [0,1]}$.
	Abusing notations, we write $\wt{X}^n = (\wt{X}^{i,n})_{i=1}^n = (\wt{X}^{n,\lambda})_{\lambda \in [0,1]}$, $\wt{Y}^n = (\wt{Y}^{i,n})_{i=1}^n = (\wt{Y}^{n,\lambda})_{\lambda \in [0,1]}$ and $\wt{Z}^n = (\wt{Z}^{i,n})_{i=1}^n = (\wt{Z}^{n,\lambda})_{\lambda \in [0,1]}$.
	Let $\Delta X_t^{i,n} = X_t^{i,n} - \wt{X}_t^{i,n}$ and $\Delta Y_t^{i,n} = Y_t^{i,n} - \wt{Y}_t^{i,n}$.

	Let us compute 
	\begin{align}\label{eq:POC-cutnorm2}
	\E [\Delta X_T^{i,n} \Delta Y_T^{i,n}]  & \geq (k-\frac{3l}{2}-\epsilon) \E[ (\Delta X^{i,n}_T)^2]- \frac{l}{2} \frac{1}{n} \sum_{j=1}^n \E[ (\Delta X^{j,n}_T)^2]  \\
	& \quad - \frac{1}{4\epsilon} \E \left[ \frac{1}{n} \sum_{j=1}^n G_n(\frac{i}{n},\frac{j}{n}) \left( \hat{Q}(\wt{X}_T^{i,n},\wt{X}_T^{j,n}) - \int_{\R} \hat{Q}(\wt{X}_T^{i,n},x)\,\cL(\wt{X}_T^{j,n})(dx) \right) \right]^2. \notag
%	& \ge (k-\frac{3l}{2}-\epsilon) \E[ (\Delta X^{i,n}_T)^2]- \frac{l}{2} \frac{1}{n} \sum_{j=1}^n \E[ (\Delta X^{j,n}_T)^2] - \frac{C}{n\epsilon}. \notag
	\end{align}
	Using It\^{o}'s formula, we also obtain that 
	\begin{align}\label{eq:POC-cutnorm3}
	&\E \left[  \Delta X^{i,n}_T \Delta Y^{i,n}_T \right] \\
	& \leq  - ( k-\frac{5l}{2}-\e) \int_0^T \E[(\Delta X^{i,n}_t)^2+ (\Delta Y^{i,n}_t )^2  ] \, dt + l \frac{1}{n} \sum_{j=1}^n \int_0^T \E[ (\Delta X^{j,n}_t)^2] \, dt \notag \\
	& + \frac{1}{4\epsilon} \int_0^T \E \left[ \frac{1}{n} \sum_{j=1}^n G_n(\frac{i}{n},\frac{j}{n}) \left( \h{B}(t,\wt{X}_t^{i,n},\wt{X}_t^{j,n},\wt{Y}_t^{i,n}) - \int_{\R} \h{B}(t,\wt{X}_t^{i,n},x,\wt{Y}_t^{i,n}) \,\cL(\wt{X}_t^{j,n})(dx) \right) \right]^2 dt  \notag \\
	& + \frac{1}{4\epsilon} \int_0^T \E \left[ \frac{1}{n} \sum_{j=1}^n G_n(\frac{i}{n},\frac{j}{n}) \left( \h{F}(t,\wt{X}_t^{i,n},\wt{X}_t^{j,n},\wt{Y}_t^{i,n}) - \int_{\R} \h{F}(t,\wt{X}_t^{i,n},x,\wt{Y}_t^{i,n}) \,\cL(\wt{X}_t^{j,n})(dx) \right) \right]^2 dt. \notag
	\end{align} 
	Combining the above two inequalities \eqref{eq:POC-cutnorm2}, \eqref{eq:POC-cutnorm3} and averaging over $i \in \{1,\dotsc,n\}$, we get that 
	\begin{align*}
	& (k-\frac{7l}{2}-\e) \left( \frac{1}{n} \sum_{i=1}^n \E[ (\Delta X^{i,n}_T)^2] + \int_0^T \frac{1}{n} \sum_{i=1}^n \E[(\Delta X^{i,n}_t)^2+ (\Delta Y^{i,n}_t )^2  ] \, dt \right) \le \frac{C}{n\varepsilon} \to 0
%	&  \leq C  \E \left[ \left| Q^{i,n}_G(X^{i,n}_T, \cL^m(X_T))-Q^{i,n}_{\wt G} ( X^{i,n}_T, \cL^m(X_T)\right|^2\right] \\
%	& +C\E \left[ \left| Q^{i,n}_G(x^{i,n}_T, \cL^m(x_T))-Q^{i,n}_{\wt G} ( x^{i,n}_T, \cL^m(x_T)\right| ^2 \right] \\
%	&+C \int_0^T \E\left[\left|B_G^{i,n}(t, \Theta^{i,n}_t)-B_{\wt G}^{i,n}(t, {\Theta}^{i,n}_t) \right|^2+\left|F_G^{i,n}(t, \Theta^{i,n}_t)-F_{\wt G}^{i,n}(t, {\Theta}^{i,n}_t) \right|^2 \right]   dt  \notag \\
%	&+C \int_0^T \E\left[\left|B_G^{i,n}(t, \theta^{i,n}_t)-B_{\wt G}^{i,n}(t, {\theta}^{i,n}_t) \right|^2+\left|F_G^{i,n}(t, \theta^{i,n}_t)-F_{\wt G}^{i,n}(t, {\theta}^{i,n}_t) \right|^2 \right]   dt .
	\end{align*}
	as $n \to \infty$.
%	Using the same argument as in \eqref{eq:cutnorm1}, we can show that the right hand side of the above inequality converges to $0$ as $\lVert G-\wt{G} \rVert_{\square} \to 0$. 
	Combining this with \eqref{eq:POC-cutnorm2} and \eqref{eq:POC-cutnorm3} gives
	$$\E[ (\Delta X^{i,n}_T)^2] + \int_0^T \E[(\Delta X^{i,n}_t)^2+ (\Delta Y^{i,n}_t )^2  ] \, dt \le \frac{C}{n} \to 0$$
	as $n \to \infty$.

	Using the argument of \textit{Step 4} in Theorem \ref{thm3.2}, we have
	$$\frac{1}{n}\sum_{i=1}^n \sum_{j=1}^n \int_0^T \E |Z^{i,j,n}_t -\delta_{ij}\wt{Z}^{i,n}_t|^2\,dt \le \frac{C}{n}.$$ 
	By the argument of \textit{Step 4} in Theorem~\ref{thm4.1}, we have the estimate \eqref{eq:POC-pf2} and hence the desired results.
	
\end{proof}

Under certain assumptions we can obtain the rate of convergence.

\begin{prop}\label{prop4.2}
	Suppose Assumption~\ref{assume:POC-2} holds.
%	 and $\lambda \mapsto \mathcal{L}(\xi^\lambda)$ is continuous with respect to $\W_{2}$. 
	Then
	\begin{align*} 
%	\label{eq:rate1}
		& \E\left[ \int_0^1  \left(  \sup_{t \in [0,T]} |X_t^{\lceil n\lambda \rceil,n}-X_t^{\lambda}|^2 + \sup_{t \in [0,T]} |Y_t^{\lceil n\lambda \rceil,n}-Y_t^{\lambda}|^2 \, ds + \int_0^T |Z^{\lceil n\lambda \rceil,\lceil n\lambda \rceil,n}_t -Z^{\lambda}_t|^2\,dt \right) \, d \ld \right] \\
		& \le \frac{C}{n} + C\|G_n - G\|_2^2 + C\int_0^1  \W_{2}^2\left(\cL(X^{\ld}_0), \cL(X_0^{\lceil n\lambda \rceil,n})\right) d\ld.
	\end{align*}
\end{prop}

\begin{proof}
	From Proposition \ref{prop3.2} we have
	\begin{align*}
	& \E\left[ \int_0^1  \left(  \sup_{u \in [0,T]} |X^{\ld}_u-\wt{X}^{n,\ld}_u|^2 + \sup_{u \in [0,T]} |Y^{\ld}_u-\wt{Y}^{n,\ld}_u|^2+ \int_0^T |Z^{\ld}_s-\wt{Z}^{n,\ld}_s|^2 \, ds \right) \, d \ld \right] \\
	& \le C\|G_n - G\|_2^2 + C\int_0^1  \W_{2}^2\left(\cL(X^{\ld}_0), \cL(\tilde{X}^{n,\ld}_0)\right) d\ld.
	\end{align*}	
	The result then follows by combining this with \eqref{eq:POC-pf2} and the observation that $\cL(\wt{X}^{n,\ld}_0)=\cL(X^{\lceil n\lambda \rceil,n}_0)$.
\end{proof}

Proposition \ref{prop4.2} provides a rate of convergence in terms of the $L^2$ convergence $\|G_n - G\|_2 \to 0$. 
Such a convergence holds for examples mentioned in Remark \ref{rmk-4.1}. 
The following is a more precise rate of convergence under Lipschitz conditions.
As mentioned in Remark \ref{rmk-4.2}, the rate in \eqref{eq:rate-2} will be slightly different if the process $X^\lambda$ is $\R^d$-valued.
 
\begin{corollary}\label{cor:4.2}
	Suppose Assumption~\ref{assume:POC-2} holds.
%	 and $\lambda \mapsto \mathcal{L}(\xi^\lambda)$ is continuous with respect to $\W_{2}$. 
	Suppose $G$ is Lipschitz continuous, $\lambda \mapsto \mathcal{L}(\xi^\lambda)$ is Lipschitz continuous with respect to $\W_{2}$, and $G_n$ is sampled from $G$, namely $G_n(\frac{i}{n},\frac{j}{n}) = G(\frac{i}{n},\frac{j}{n})$.
	Then $\|G_n - G\|_2 \le \frac{C}{n}$ and hence \eqref{eq:rate-0}--\eqref{eq:rate-2} hold.
\end{corollary}

\begin{proof}
	The proof is similar to that of Corollary \ref{cor:4.1}, except that the use of Proposition \ref{prop4.1} is replaced by Proposition \ref{prop4.2}, and hence omitted.
\end{proof}

\section{Graphon mean field game and convergence of $n$-player game}
Let $G: [0,1] \times [0,1] \to \R_+$ be a bounded graphon, and without loss of generality assume that $|G(\ld,\k)| \leq 1, \, \forall (\ld,\k) \in [0,1]^2$. Each $\lambda \in [0,1]$ represents a type of population, which consists of continuum many players. Let $\bs{\eta}^{\lambda} \in \cP_2(\R)$ denote the distribution of population of type $\lambda$, and $\bs{\eta}$ denote the collection $\{\bs{\eta}^{\lambda}: \, \lambda \in [0,1] \}$, i.e., $\bs{\eta} \in \cM([0,1]; \cP_2(\R))$. Let $A \subset \R^n$ be a  convex control space. Take functions 
\begin{align*}
b_1, f_1, f_2&: [0,T] \times\R \to \R,\\
b_2,b_3&: [0,T] \to \R, \\
q_1, q_2 &: \R \to \R. 
\end{align*}
For any $(\lambda,x,\bs{\eta},a) \in [0,1] \times \R  \times \cM([0,1]; \cP_2(\R) )\times A$, we define 
\begin{align*}
b^{\ld}_G(t,x,\bs{\eta}, a)&:= \int_{[0,1]} G(\ld,\kappa) \, d \kappa \int_{\R} b_1(t,z) \, \bs{\eta}^{\kappa}(dz) + b_2(t)x+ b_3(t)a, \\
f^{\ld}_G(t,x,\bs{\eta},a)&:= f_1(t, x)+ \int_{[0,1]} G(\ld, \ka) \, d \kappa \int_{\R } f_2(t, z) \, \bs{\eta}^{\kappa}(dz) +\frac{1}{2} a^2,\\
q^{\ld}_G(x, \bs{\eta})&:= q_1(x)+ \int_{[0,1]} G(\ld,\ka) \, d \kappa \int_{\R} q_2(z)\,  \bs{\eta}^{\ka}(dz) .
\end{align*}

Let $W^{\lambda}$ be a family of independent standard Brownian motion, and $\sigma>0$ be a constant volatility. Denote by $\bs{\mu}_G^{\lambda}(t)$ the distribution of players of type $\lambda$ at time $t$, and $\bs{\mu}_G(t): = \{ \bs{\mu}^{\lambda}_G(t): \, \lambda \in [0,1] \}$, $\bs{\mu}_G:=\{\bs{\mu}_G(t): \, t \in [0,T] \}$.   Choosing $\a^{\lambda}(t) \in A$, a representative player of type $\lambda$ controls the dynamic 
\begin{align*}
\begin{cases}
dX^{\lambda}_t = {b}_G^{\lambda}\left(t,X^{\lambda}_t,\bs{\mu}_G(t), \a^{\lambda}(t)\right)\, dt + \sigma \, d W^{\lambda}_t, \\
X^{\lambda}_0 =\xi^{\lambda},
\end{cases}
\end{align*}
where $\xi^{\lambda}$ is a square integrable random variable. The cost for the representative player $\lambda$ is given by
\begin{align}\label{eq:cost} 
J(\a^{\lambda}, \bs{\mu}_G) =\E \left[\int_0^T {f}_G^{\lambda}\left(t,X^{\lambda}_t,  \bs{\mu}_G(t),\a^{\lambda}(t)\right) \, dt + {q}_G^{\lambda}\left(X^{\lambda}_T, \bs{\mu}_G(T)\right)\right],
\end{align}
and each representative player chooses control $\a^{\ld}(t)$ to minimize $J(\a^{\ld}, \bs{\mu}_G)$. 

For each $\lambda \in [0,1]$, define the Hamiltonian 
\begin{align*}
H_G^{\lambda}\left(t,x,\bs{\mu}_G(t),y,a\right):= {b}_G^{\lambda} \left(t,x, \bs{\mu}_G(t),a\right) \cdot y + {f}_G^{\lambda}\left(t,x,a,\bs{\mu}_G(t)\right),
\end{align*}
and the minimizer 
\begin{align*}
\hat{\a}^{\lambda}_G(t,x, \bs{\mu}_G(t), y):= \lambdargmin_{a \in A} H_G^{\lambda}\left(t,x,\bs{\mu}_G(t),y,a\right)=-b_3(t)y . 
\end{align*}
 Given $\{\bs{\mu}_G(t): \, 0 \leq t \leq T\}$, by Pontryagin's maximum principle, we obtain a family of BSDE 
\begin{align}\label{eq:BSDE}
\begin{cases}
dY^{\lambda}_t = -\pa_x H_G^{\lambda}\left(t,X^{\lambda}_t, \bs{\mu}_G(t), Y^{\ld}_t,-b_3(t)Y^{\ld}_t  \right) dt +Z^{\lambda}_t \, d W^{\lambda}_t, \\
Y^{\lambda}_T =\pa_x{q}_G^{\lambda}\left(X^{\lambda}_T, \bs{\mu}_G(T)\right).
\end{cases} 
\end{align}

Since the law of $X^{\lambda}_t$ should coincide with $\bs{\mu}^{\lambda}_G(t)$, after simplification we obtain the FBSDE of the graphon field game 
\begin{align}\label{eq:FBSDE}
\begin{cases}
dX^{\lambda}_t = \left( b_2(t)X^{\lambda}_t  -|b_3(t)|^2 Y^{\lambda}_t + \int_0^1 \int_\R G(\lambda,\kappa) \, \E[ b_1(t,X^{\kappa}_t )] \,d\kappa  \right)  dt + \sigma \, d W^{\lambda}_t, \\
dY^{\lambda}_t =- \left(b_2(t) Y^{\ld}_t+\pa_x f_1(t,X^{\ld}_t) \right) \,dt  +Z^{\lambda}_t \, d W^{\lambda}_t, \\
 X^{\lambda}_0 =\xi^{\lambda}, \\
 Y^{\lambda}_T = \pa_x q_1 (X_T^{\lambda}), \quad \forall \ld \in [0,1].
\end{cases}
\end{align}

In order for coefficients of \eqref{eq:FBSDE} to satisfy Assumption~\ref{assume1}, \ref{assume3}, \ref{assume:POC-1} or \ref{assume2}, \ref{assume3}, \ref{assume:POC-2}, we propose the following conditions.

\begin{assume}\label{assume4}
(i)  $b_1$ grows at most linearly $x$. $b_2(t), b_3(t)$ are uniformly bounded. $f_1, f_2,q_1$ are differentiable, and of at most quadratic growth in $x$. \\
(ii) $f_1,f_2, q_1 $ are convex in $x$. \\
(iii)  $b_1,\pa_x f_1, \pa_x f_2, \pa_x q_1 $ are $L$-Lipschitz in $x$, and $\max_{ t \in [0,T]} |b_3(t)| \leq L$ for some $L>1$. \\
(iv)  It holds that
\begin{align}\label{eq:coeffinq1}
\max_{t \in [0,T]} b_2(t) < -100 L^4.
\end{align}
 (v) We have $\sup_{\ld \in [0,1]} \E[|\xi^{\ld}|^p] < +\infty$ and $\ld \to \cL(\xi^{\ld})$ is continuous with respect to $\W_2$.
\end{assume}

Due to the explicit structure of \eqref{eq:FBSDE}, we can easily check that its coefficients satisfy Assumptions~\ref{assume1}, \ref{assume3} and \ref{assume:POC-1}. Therefore we obtain the following result. 
\begin{corollary}
Under Assumption~\ref{assume4}, there exists a unique solution to \eqref{eq:FBSDE}, the solution is stable in the sense of Theorem~\ref{thm3.1}, Proposition~\ref{prop3.1}, and the propagation of chaos results hold as in Theorem~\ref{thm4.1}, Proposition~\ref{prop4.1}, Corollary~\ref{cor:4.1}. 
\end{corollary}

\begin{assume}\label{assume5}
(i)  $b_1$ grows at most linearly $x$. $b_2(t), b_3(t)$ are uniformly bounded. $f_1, f_2, q_1$ are differentiable, and of at most quadratic growth in $x$. \\
(ii)  There exist positive $\iota$ such that 
\begin{align*}
q_1(x')-q_1(x)-(x'-x)\pa_x q_1(x) & \geq \iota (x'-x)^2 ,
\end{align*}
and $f_1, f_2$ are convex in $x$. \\
(iii)  $b_1$ and $ \pa_x f_1$ are $L$-Lipschitz in $x$ for some $L \geq 1$.\\
(iv)  It holds that 
\begin{align}\label{eq:coeffieq2}
 \min\left\{\inf_{t \in [0,T]} |b_3(t)|^2, \iota \right\} \geq 100 L^2.
\end{align}
 (v) We have $\sup_{\ld \in [0,1]} \E[|\xi^{\ld}|^p] < +\infty$ and $\ld \to \cL(\xi^{\ld})$ is continuous with respect to $\W_2$.
\end{assume}

\begin{corollary}
Under Assumption~\ref{assume5}, there exists a unique solution to \eqref{eq:FBSDE}, the solution is stable in the sense of Theorem~\ref{thm3.2}, Proposition~\ref{prop3.2}, and the propagation of chaos results holds as in Theorem~\ref{thm4.2}, Proposition~\ref{prop4.2}, Corollary~\ref{cor:4.2}. 
\end{corollary}

Now let us we turn to the convergence of finite player game. Fix $n \in \N$ and $G_n$. For any $(i, \ul{x}, a^i) \in \{1, \dotso, n\} \times \R^n \times A^n$, we define 
\begin{align*}
b^{i,n}(t,\ul{x}, a^i)&:= \frac{1}{n}\sum_{j=1}^n G_n(i,j) b_1(t,x^j) + b_2(t)x^i+ b_3(t)a^i, \\
f^{i,n}(t,\ul{x},a^i)&:= f_1(t, x^i)+ \frac{1}{n}\sum_{j=1}^n G_n(i,j) f_2(t, x^j)+\frac{1}{2}(a^i)^2, \\
q^{i,n}(\ul{x})&:= q_1(x^i)+ \frac{1}{n} \sum_{j=1}^n G_n(i,j) q_2(x^j).
\end{align*}
Let us compute the FBSDE system of this $n$-player game. Each player has the Hamiltonian 
\begin{align*}
H^{i,n}(t,\ul{x}, \ul{y}^i, \ul{a})=b^n(t,\ul{x}, \ul{a}) \cdot \ul{y}^i + f^{i,n} (t,\ul{x}, \ul{a}^i),
\end{align*}
where $\ul{x}, \ul{y}^i \in \R^n$, $\ul{a} \in A^n$. By our construction, it is clear that for any $t \geq 0$, $ \ul{x} \in \R^n$ and $\ul{y} \in \R^{n \times n}$, functions $\h{\a}^{i,n}(t, \ul{x},\ul{y}^{i}) :=-b_3(t) y^{i,i}$ satisfy that 
\begin{align*}
H^{i,n}(t,\ul{x},\ul{y}^i, \h{\a}^n(t, \ul{x},\ul{y})) \leq H^{i,n} (t, \ul{x},\ul{y}^i, (a^i, \h{\a}^n(t, \ul{x},\ul{y})^{-i})),
\end{align*}
for all $a^i \in A$.

Note that when $i=j$, we have 
\begin{align*}
\pa_{x^i} H^{i,n}(t,\ul{x}, \ul{y}^i, \ul{\h{\a}}^n(t,\ul{x}, \ul{y}))=& b_2(t) y^{i,i} +\frac{1}{n} \sum_{k=1}^n G_n(k,i) \pa_x b_1(t,x^i) y^{i,k}+ \pa_{x} f_1(t,x^i) \\
&+ \frac{1}{n}G_n(i,i) \pa_x f_2(t,x^i), \\
\pa_{x^i} q^{i,n}(\ul{x})=&\pa_x q_1(x^i) + \frac{1}{n} G_n(i,i) \pa_x q_2(x^i), 
\end{align*}
and when $i \not =j$, 
\begin{align}\label{eq:qexplicit} 
\pa_{x^j} H^{i,n}(t,\ul{x}, \ul{y}^i, \ul{\h{\a}}^n(t,\ul{x}, \ul{y}))= & b_2(t) y^{i,j} +\frac{1}{n} \sum_{k=1}^n G_n(k,j) \pa_x b_1(t,x^j) y^{i,k}+ \frac{1}{n}G_n(i,j) \pa_x f_2(t,x^j), \notag \\
\pa_{x^j} q^{i,n}(\ul{x})=&\frac{1}{n} G_n(i,j) \pa_x q_2(x^j).
\end{align}

We obtain the FBSDE system for the $n$-player game
\begin{align}\label{eq:nplayer}
\begin{cases}
dX^{i,n}_t=b^{i,n}(t,\ul{X}^n_t, \ul{\h{\a}}^{i,n}(t,\ul{X}^n_t, \ul{Y}^{i,n}_t))\,dt + \sigma \, dW^{i/n}_t, \\
dY^{i,j,n}_t=-\pa_{x^j} H^{i,n}(t,\ul{X}^n_t, \ul{Y}_t^{i,n}, \ul{\h{\a}}^n(t,\ul{X}^n_t,\ul{Y}^n_t)) \, dt +\sum_{k=1}^n Z^{i,j,k,n}_t \, dW^{k/n}_t, \\ 
X^{i,n}_0= \xi^{i/n}, \\
Y^{i,j,n}_T=\pa_{x^j} q^{i,n}(\ul{X}^n_T), \quad i,j=1,\dotso, n.
\end{cases}
\end{align}

We briefly show the convergence result 
\begin{align*}
\frac{1}{n} \sum_{i=1}^n  \E\left[ \sup_{t \in [0,T]} |X_t^{i,n}-X_t^{i/n}|^2 +\sup_{t \in [0,T]} |\h{\a}^{i,n}_t-\h{\a}_t^{i/n}|^2\right]t \to 0  \ \  \text{as $n \to \infty$},
\end{align*}
which in our model is equivalent to 
\begin{align}\label{eq:convergence}
\frac{1}{n} \sum_{i=1}^n  \E\left[\sup_{t \in [0,T]} |X_t^{i,n}-X_t^{i/n}|^2 +\sup_{t \in [0,T]} |Y^{i,i,n}_t-Y_t^{i/n}|^2\right]   \to 0  \ \  \text{as $n \to \infty$}. 
\end{align}
The argument is divided into two steps. 
\begin{thm}\label{eq:thm5.1}
Under Assumption~\ref{assume4}, if $\lVert G_n - G \rVert_{\square} \to 0$ and $G$ is continuous, then the Nash equilibrium of $n$-player game converges to the corresponding graphon field game, i.e., \eqref{eq:convergence} holds. 
\end{thm}
\begin{proof}
\textit{Step 1:} Consider an auxiliary FBSDE system 
\begin{align}
\begin{cases}
d\wt{X}^{i,n}_t = \left( b_2(t)\wt{X}^{i,n}_t  -|b_3(t)|^2 \wt{Y}^{i,n}_t + \frac{1}{n}\sum_{k=1}^n G_n(i,k) b_1(t,\wt{X}^{k,n}_t ) \right)  dt + \sigma \, d W^{i/n}_t, \\
d\wt{Y}^{i,n}_t =- \left(b_2(t) \wt{Y}^{i,n}_t+\pa_x f_1(t,\wt{X}^{i,n}_t) \right) \,dt  +\wt{Z}^{i,j,n}_t \, d W^{i/n}_t, \\
\wt{X}^{i,n}_0 =\xi^{i/n}, \\
\wt{Y}^{i,n}_T = \pa_x q_1 (\wt{X}_T^{i,n}),
\end{cases} 
\end{align}
Invoking Theorem~\ref{thm4.1},  we obtain that 
\begin{align*}
\frac{1}{n} \E\left[\sup_{t \in [0,T]} |\wt{X}_t^{i,n}-X_t^{i/n}|^2 +\sup_{t \in [0,T]} |\wt{Y}^{i,n}_t-Y_t^{i/n}|^2\right]  \to 0  \ \  \text{as $n \to \infty$}. 
\end{align*}

\vspace{5pt} 
\textit{Step 2:} We will show that 
\begin{align}\label{eq:goal}
\frac{1}{n}  \E\left[\sup_{t \in [0,T]} |\wt{X}_t^{i,n}-X_t^{i,n}|^2 +\sup_{t \in [0,T]} |\wt{Y}^{i,n}_t-Y_t^{i,n}|^2\right]  \to 0  \ \  \text{as $n \to \infty$}. 
\end{align}
 Using the same computation as in Theorem~\ref{thm:contraction}, we can prove that 
\begin{align*}
\frac{1}{n} \E \left[ \sum_{i=1}^n \int_0^T |X_t^{i,n}|^2 + \sum_{j=1}^n |Y^{i,j,n}_t |^2 \, dt \right] \leq C,
\end{align*}
where $C$ is some constant uniformly for any $n \in \N$. Using this bound, the same computation shows that for any $i \in \{1, \dotso, n\}$ 
\begin{align*}
\E\left[\int_0^T |X^{i,n}_t|^2 + \sum_{j=1}^n |Y^{i,j,n}_t|^2 \, dt  \right] \leq C,
\end{align*}
and also  
\begin{align}\label{eq:uniformbound}
\E \left[|X^{i,n}_t|^2 + \sum_{j=1}^n |Y^{i,j,n}_t|^2\right] \leq C. 
\end{align}
According to \eqref{eq:qexplicit}, for any $i \not = j$ we have the terminal $Y_T^{i,j,n}= \frac{1}{n} G_n(i,j) \pa_x q_2(X^{j,n}_T)$ and its drift  $|-\pa_{x^j} H^{i,n}| \leq C(|y^{i,j,n}| + \frac{1}{n} \sum_{k \not = i} |y^{i,k,n}|)+\mathcal{O}(1/n)$. Therefore, one can obtain 
\begin{align}\label{eq:smallY}
\E\left[|Y^{i,j,n}_t|^2 \right] \leq \frac{C}{n} \quad \text{for any $n \in \N$ and some $C>0$}
\end{align}
as in \cite[Lemma 22]{2020arXiv200408351L}.

Using monotonicity conditions, and computing as in Theorem~\ref{thm:contraction}, it can be seen that 
\begin{align*}
& \frac{1}{n} \int_0^T  \E\left[ |\wt{X}_t^{i,n}-X_t^{i,n}|^2 + |\wt{Y}^{i,n}_t-Y_t^{i,i,n}|^2\right] dt \\
& \leq \frac{C}{n} \sum_{i=1}^n \left(\frac{1}{n}  G_n(i,i) \E \left[| \pa_x q_2(X^{i,n}_T)|^2 +\int_0^T |\pa_x f_2(t,X^{i,n}_t)|^2 \, dt\right] \right) \\
& \ \ \ \ + \frac{C}{n^2} \sum_{i,k=1}^n G_n(k,i) \E\left[ \int_0^T |\pa_x b_1(t,X^{i,n}_t) Y_t^{i,k,n}|^2 \,dt  \right] \\
& \leq C/n,
\end{align*}
where we use \eqref{eq:uniformbound} and \eqref{eq:smallY}. Then by a similar argument as in \emph{Step 4} of Theorem~\ref{thm3.1}, one can easily conclude \eqref{eq:goal}.

\end{proof}

\begin{remark}
We want to point out that we are only able to prove the convergence under Assumption~\ref{assume4}. It is just because it is difficult to show the uniform boundedness of solutions of \eqref{eq:nplayer} under Assumption~\ref{assume5}. The rest of the proof actually works for both assumptions. 
\end{remark}

\appendix

\section{Measurability}

\begin{lemma}\label{lem1}
$\bs{\mu}: [0,1] \to \cC\left([0,T]; \cP_p(\R) \right)$ is measurable if and only if for any $t \in [0,T]$, 
\begin{align}\label{lem1:eq1}
\lambda \mapsto \bs{\mu}^{\lambda}(t) \in \cP_p(\R) \text{ is measurable}.
\end{align}
\end{lemma}
\begin{proof}
The proof of `only if' is trivial. For the proof of `if' part, we note that with sup norm,  $\cC\left([0,T];\cP_p(\R)\right)$ is a topological subspace of $ \cC^0 \left([0,T];\cP_p(\R)\right)$.

For any $n \in \N$, due to \eqref{lem1:eq1} we know that 
$$\lambda \mapsto (\bs{\mu}^{\lambda}(T/n), \dotso, \bs{\mu}^{\lambda}(T) ) \text{ is measurable}.$$
We construct $\bs{\mu}_n \in \cM\left([0,T];\cP_p(\R)\right)$
\begin{align*}
\bs{\mu}_n^{\lambda}(t):= \bs{\mu}^{\lambda}\left(\frac{i}{n}\right), \, t \in \left(\frac{(i-1)T}{n},\frac{iT}{n} \right], \, i=1, \dotso, n.
\end{align*}
Then it can be easily verified that $\lambda \to \bs{\mu}_n^{\lambda}$ is measurable. By the continuity of $\bs{\mu}^{\lambda}(\cdot)$, $\lim\limits_{n \to \infty} \bs{\mu}_n^{\lambda}= \bs{\mu}^{\lambda}$ is measurable in $\lambda$. 
\end{proof}

\begin{lemma}\label{lem3}
A function $x: \ld \mapsto x^{\ld} \in L^{p,c}_{\mathcal{F}}$ belongs to $\cM L^{p,c}_{\mathcal{F}}$ if and only if $ \ld \mapsto x^{\ld}_t \in L^p_{\mathcal{F}_t}$ is measurable for any $t \in [0,T]$. 
\end{lemma}
\begin{proof}
Note that  $L^{p,c}_{\mathcal{F}} \ni \bs{\eta} \mapsto \bs{\eta}_t \in L^p_{\mathcal{F}_t}$ is continuous. Therefore it can be readily seen that the measurability of $\ld \mapsto x^{\ld}$ implies the measurability of $\ld \mapsto x^{\ld}_t$ for any $t \in [0,T]$. 

Conversely, define $x^{N} \in \cM L^{p,2}_{\mathcal{F}}$ for $N \in \N$ as follows, 
\begin{align*}
x^{N,\ld}_t&:=x^{\ld}_{nT/N}, \quad \forall t \in [nT/N, (n+1)T/N), \, n =0,\dotso, N-2, \, \ld \in [0,1], \\
x^{N,\ld}_t&:=x^{\ld}_{(N-1)T/N}, \quad \forall t \in [(N-1)T/N,T], \, \ld \in [0,1]. 
\end{align*}
According to our hypothesis, it can be easily seen that 
\begin{align*}
\ld \mapsto x^{N, \ld} \in \left(L^{p,2}_{\mathcal{F}}, \wt{ \lVert \cdot \rVert}_S^p\right)
\end{align*}
is measurable, and also the limit
\begin{align*}
\ld \mapsto x^{\ld} \in \left(L^{p,c}_{\mathcal{F}}, \wt{ \lVert \cdot \rVert}_S^p \right).
\end{align*}
\end{proof}

\begin{lemma}\label{lem:measurable}
Take a polish space $\Omega$ and a Borel probability measure $(\mathcal{F},P)$ over $\Omega$. Take another measure space $(E, \Sigma, m)$.  Suppose $\rho: E \times \R \to \R$ is a real-valued function such that $x \mapsto \rho (e,x)$ is continuous for any $e \in E$, $e \mapsto \rho(e,x)$ is measurable for any $x \in \R$, and $|\rho(e,x)| \leq C(1+|x|),\, \forall (e,x) \in E \times \R$ for some positive constant $C$. Then given any measurable mapping $e \mapsto X(e) \in L^p(\Omega, \mathcal{F},P)$, the Banach-valued function $$ e \mapsto \rho(e, X(e)) \in L^p(\Omega, \mathcal{F},P)$$ is also measurable.

\end{lemma}
\begin{proof}
According to \cite[Proposition 3.4.5]{MR3098996}, the Banach space $L^p(\Omega, \mathcal{F}, P)$ is separable. Therefore as a result of Pettis measurability theorem, any measurable function $X: E \to L^p(\Omega, \mathcal{F}, P)$ is also strongly measurable, i.e., $X$ can be written as a pointwise limit of simple functions $$X^n= \sum_{i=1}^{m_n}\mathbbm{1}_{S_{m_n}} x_{m_n},$$ where $m_n \in \N$, $S_1, \dotso, S_{m_n}$ is a finite collection of disjoint subsets of $E$, and $x_1, \dotso, x_{m_n} \in L^p(\Omega, \mathcal{F}, P)$. It is then readily seen that 
\begin{align*}
e \mapsto \rho(e, X^n(e))
\end{align*}
is measurable, and thus $\rho(\cdot, X(\cdot))= \lim\limits_n \rho(\cdot, X^n(\cdot))$ is measurable. 
\end{proof}

\begin{lemma}\label{lem:measurable'}

Suppose $\rho: [0,1] \times [0,T] \times \R \to \R$ is a measurable function such that $x \mapsto \psi(\ld ,t ,x)$ is continuous and grows at most linearly uniformly for $(\ld ,t) \in [0,1] \times [0,T]$. Given any measurable $\ld \mapsto X^{\ld} \in L^{p,c}_{\mathcal{F}}$, we have that 
\begin{align}\label{eq:measurableintegral}
\ld \mapsto  \int_0^{\cdot} \psi(\ld, s ,X^{\ld}(s)) \, ds  \in L^{p,c}_{\mathcal{F}}
\end{align}
is measurable. 
\end{lemma}
\begin{proof}
By our assumption, it is clear that $(\ld ,s) \mapsto X^{\ld}(s)$ is measurable. Applying Lemma~\ref{lem:measurable} with $E=[0,1] \times [0,T]$, it is readily seen that 
\begin{align}\label{eq:appendix1}
(\ld ,s) \mapsto \psi(\ld, s ,X^{\ld}(s))
\end{align}
is measurable. The function \eqref{eq:appendix1} is also Bochner integrable due to our linear growth assumption in $x$. Thus by the Fubini theorem of Bochner theorem, 
$$\ld \mapsto  \int_0^{t} \psi(\ld, s ,X^{\ld}(s)) \, ds$$
is measurable for any $t \in [0,T]$. Now the measurable of \eqref{eq:measurableintegral} follows from Lemma~\ref{lem3}.

\end{proof}

\begin{remark}
Using approximation of simple functions, one can easily verify that Bochner integral coincides with Lebesgue integral.
\end{remark}

\section{Weak uniqueness of FBSDE}
The notion of weak existence and uniqueness for FBSDEs are almost the same to the ones considered for classical SDEs, see e.g. \cite{doi:10.1081/SAP-120020423,DELARUE20061712,10.1214/08-AOP0383}.

\begin{definition}
A five-tuple $(\Omega, \mathcal{F}, \mathbb{F}, P, W)$ is said to be a standard set-up if $W$ is a Brownian motion over the probability space $(\Omega, \mathcal{F},\mathbb{F},P)$ and $\mathbb{F}:=\{ \mathcal{F}\}_{t \geq 0}$ is  complete and right continuous. 
\end{definition}

Consider an FBSDE 
\begin{align}\label{eq:weak}
\begin{cases}
X_t=x+\int_0^t B(s,X_s, Y_s) \, ds + \sigma \, W_t, \\
Y_t= Q(X_T)+\int_t^T F(s,X_s,Y_s) \, ds- \int_t^T Z_s \, d W_s, 
\end{cases}
\end{align}
where $B,F,Q$ are progressively measurable functions.

\begin{definition}
A triple of processes $(X,Y,Z)$ is said to be a weak solution of \eqref{eq:weak} if there exists a standard set-up $(\Omega, \mathcal{F}, \mathbb{F}, P, W)$ such that $(X,Y,Z)$ are adapted to the filtration $\mathbb{F}$ and satisfy \eqref{eq:weak} a.s. If $(X,Y,Z)$ and $(\wt{X}, \wt{Y},\wt{Z})$ are two weak solutions of \eqref{eq:weak} on the same set-up, we say that pathwise uniqueness holds if 
\begin{align*}
\mathbb{P} \left[ (X_t,Y_t)=(\wt{X}_t,\wt{Y}_t), \, \forall t \in [0,T] \right]=1.
\end{align*}
\end{definition}
By Yamada-Watanabe Theorem for SDEs, pathwise uniqueness implies uniqueness in law. We have the same result for FBSDEs.
\begin{lemma}\label{lem:unique}
Suppose the pathwise uniqueness property holds for FBSDE \eqref{eq:weak}. Then for any two weak solutions $(X,Y,Z)$ on $(\Omega, \mathcal{F}, \mathbb{F}, P, W)$ and $(\wt{X},\wt{Y},\wt{Z})$ on $(\wt{\Omega}, \wt{\mathcal{F}}, \wt{\mathbb{F}}, \wt{P}, \wt{W})$, their distributions coincide.
\end{lemma}
\begin{proof}
See \cite[Theorem 5.1]{doi:10.1081/SAP-120020423}.
\end{proof}

\bibliographystyle{siam}
\bibliography{ref}
\end{document}